\documentclass[reqno,centertags,12pt]{amsart}

\usepackage{amsmath}
\usepackage{amscd}
\usepackage{stackrel}
\usepackage{amssymb}
\usepackage{amsthm}
\usepackage{bbm}
\usepackage{latexsym}
\usepackage{mathrsfs}
\usepackage{verbatim}
\usepackage{tikz-cd}
\usepackage{mathtools}
\usepackage[hidelinks]{hyperref}

\usepackage[shortlabels]{enumitem}

\usepackage{xcolor, cite}

\usepackage{tikz}
\pgfdeclarelayer{nodelayer}
\pgfdeclarelayer{edgelayer}
\pgfsetlayers{edgelayer,nodelayer,main}
\tikzstyle{arrow}=[draw=black,arrows=-latex]

-\marginparwidth \oddsidemargin -4mm \evensidemargin \oddsidemargin

\newtheorem{theorem}{Theorem}[section]

\newtheorem{proposition}[theorem]{Proposition}
\newtheorem{lemma}[theorem]{Lemma}
\newtheorem{corollary}[theorem]{Corollary}
\theoremstyle{definition}
\newtheorem{definition}[theorem]{Definition}

\newcounter{smalllist}

\DeclareMathOperator*{\sgn}{sgn}

\allowdisplaybreaks
\numberwithin{equation}{section}

\newcommand{\abs}[1]{\left\lvert#1\right\rvert}

\newcommand{\norm}[1]{\left\|#1\right\|}

\newcommand{\set}[1]{\left\{ #1 \right\}}

\newcommand{\lb}{\label}

\newcommand{\ess}{\text{\rm{ess}}}

\newcommand{\beq}{\begin{equation}}
\newcommand{\eeq}{\end{equation}}

\newcommand{\bal}{\begin{align}}
\newcommand{\eal}{\end{align}}
\newcommand{\bals}{\begin{align*}}
\newcommand{\eals}{\end{align*}}

\newcommand{\bbN}{{\mathbb{N}}}
\newcommand{\bbR}{{\mathbb{R}}}

\newcommand{\bbZ}{{\mathbb{Z}}}

\newcommand{\bbQ}{{\mathbb{Q}}}
\newcommand{\bbT}{{\mathbb{T}}}

\newcommand{\calL}{{\mathcal L}}

\newcommand{\eps}{\varepsilon}

\newcommand{\tht}{\theta}

\begin{document}
\title[Low Regularity Solutions to the g-SQG Equation with Regular Level Sets]
{Well-Posedness for Low Regularity Solutions to the g-SQG Equation with Regular Level Sets}

\author{Junekey Jeon and Andrej Zlato\v{s}}

\address{\noindent Department of Mathematics \\ University of
California San Diego \\ La Jolla, CA 92093 \newline Email: \tt
zlatos@ucsd.edu,
j6jeon@ucsd.edu}

\begin{abstract}
We show that the generalized SQG equation on the plane  is locally well-posed in spaces of low regularity solutions (essentially H\" older continuous with H\" older  exponents depending on the equation parameter $\alpha\in(0,\frac 12)$) that have $H^2$ level sets (i.e., with $L^2$ curvatures). Moreover, for $\alpha\le\frac 16$ and initial data satisfying some additional hypotheses we show that the corresponding solutions can stop existing only when their level sets lose $H^2$-regularity, and hence not just due to level set collisions or ``pile ups''.
\end{abstract}

\maketitle

\section{Introduction}

The {\it generalized surface quasi-geostrophic equation (g-SQG)} is the active scalar PDE
\begin{equation}\label{1.1}
    \partial_{t}\theta + u(\tht)\cdot\nabla\theta = 0
\end{equation}
on $I\times \bbR^2$ for some open time interval $I$, where the transporting velocity is given by
\begin{equation}\label{1.2}
    u(\tht) \coloneqq -\nabla^{\perp}(-\Delta)^{-1+\alpha}\theta
\end{equation}
with $\alpha\in\left(0,\frac{1}{2}\right)$, and we denote
$\nabla^{\perp}\coloneqq(-\partial_{x_{2}},\partial_{x_{1}})$.  When $\alpha=0$ and $\alpha=\frac 12$, this becomes the {\it Euler} and {\it surface quasi-geostrophic equation (SQG)}, respectively, with $\tht=\nabla^\perp\cdot u(\theta)$  being the vorticity of $u(\theta)$ in the former case.
Denoting $(x_{1},x_{2})^{\perp}\coloneqq(-x_{2},x_{1})$, we can write \eqref{1.2} as 
\beq\lb{11.7}
    u(\theta^t;x) \coloneqq   c_{\alpha} \int_{\bbR^{2}}
    \frac{(x-y)^\perp}{\abs{x-y}^{2+2\alpha}}\, \theta^t(y)\,dy =  \int_{\bbR^{2}}    \nabla^{\perp}K(x - y)\,\theta^t(y)\,dy
\eeq
with $\theta^t:=\theta(t,\cdot)\in L^{1}(\bbR^{2})\cap L^{\infty}(\bbR^{2})$ at any time $t\in I$, where $c_{\alpha}>0$ is some constant
%the velocity field $u(\theta)$ generated by $\theta$ as in 
%where the kernel $K\colon \bbR^{2}\to(0,\infty]$ is defined as
and
\[
    K(x)\coloneqq - \frac{c_{\alpha}}{2\alpha\abs{x}^{2\alpha}} .
%    \qquad\text{and}\qquad 
%    \nabla^{\perp}K(x) = \frac{c_{\alpha}(x_2,-x_1)}{\abs{x}^{2+2\alpha}}.
\]

The g-SQG equation first appeared in physics and atmospheric science literature (see, e.g.,  \cite{PHS, Smith}), with the first mathematical studies of it, including local regularity in sufficiently smooth spaces of functions, being undertaken in \cite{CCW,CIW}.
Note that since \eqref{1.1} is a transport PDE, it can in principle also be solved in low regularity spaces in both weak \cite{Nguyen} and Lagrangian (i.e., transport) sense.  The latter notion of solutions typically also guarantees the former, and we will employ it here, too (see Definition \ref{D.11.1} below).

While $u(\theta^t)$ is clearly a $(1-2\alpha)$-H\"{o}lder continuous vector field when $\theta^t\in L^{1}(\bbR^{2})\cap L^{\infty}(\bbR^{2})$, this is not sufficient to obtain well-posedness for the PDE in this space.
%$L^{1}(\bbR^{2})\cap L^{\infty}(\bbR^{2})$.  
However, if also $\theta^t\in C^{\beta}(\bbR^{2})$ for some $\beta>2\alpha$, then $u(\theta^t)$ is Lipschitz and this easily yields local well-posedness in $L^{1}(\bbR^{2})\cap  C^{\beta}(\bbR^{2})$ for any  $\beta>2\alpha$  \cite{CJK}.  We note that recently local well-posedness \cite{JKY, ZlaSQGsing} and finite time singularity formation \cite{ZlaSQGsing} were proved for \eqref{1.1} on the half-plane in appropriate spaces of continuous functions with regularity deteriorating at the boundary.

On the other hand, \eqref{1.1}--\eqref{1.2} is also locally well-posed in the space of (spatially discontinuous) {\it patch solutions} of the form
\[
\tht^t =\sum_{n=1}^N \mu_n \mathbbm{1}_{\Theta_n^t}
\]
at any time $t$, with some constants $\mu_n\in\bbR\setminus\{0\}$ and bounded {\it patches} $\Theta_n^t\subseteq\bbR^2$ that are transported by the velocity $u(\theta)$, provided all the initial patch boundaries $\partial \Theta_n^0$ are disjoint and sufficiently smooth simple closed curves \cite{Gancedo,GanPat}.  The reason for this is that the dynamic of these solutions is fully determined by $u(\theta)$ restricted to the patch boundary curves, and as long as these are smooth enough at all times $t$, the same is true for the velocity restricted to them (because near each patch boundary, $\tht^t$ can be seen as constant in the tangential direction).
In addition, blow-up criteria \cite{JeoZla, KisLuo,GanStr} and ill-posedness results \cite{KisLuo2} were recently obtained for this {\it g-SQG patch problem}, along with local well-posedness and finite time singularity results on the half-plane \cite{KisYaoZla,         KRYZ,GanPat} and when some of the  patches are allowed to touch each other \cite{JeoZla2} (see below for more details on this).

Since  patch boundaries are  also the boundaries of the level sets of patch solutions, the above explanation of existence of local regularity results for the g-SQG patch problem suggests a natural question:  {\it Is there a local well-posedness theory for \eqref{1.1}--\eqref{1.2} in spaces of low regularity continuous solutions with sufficiently regular level sets?}  

In our main result, Theorem \ref{T1.4} below, {we answer this in the affirmative} in the space of functions whose level sets (or their boundaries) are  simple closed $H^2$  curves and they belong to $L^{1}(\bbR^{2})\cap C_\rho(\bbR^2)$ for some increasing continuous  $\rho:[0,\infty)\to[0,\infty)$ satisfying  $\rho(0)=0$ and 
\beq\lb{11.5}
\int_{0}^{1}\frac{\rho(s)}{s^{1+2\alpha}}\,ds < \infty.
\eeq
  Here $f\in C_\rho(\bbR^2)$ whenever $f$ has modulus of continuity $a\rho$ for some $a>0$, where the latter means that $|f(x)-f(y)| \le a\rho(|x-y|)$ for all $x,y\in\bbR^2$. For instance, we can take $\rho(s):=s^{2\alpha} \max\{-\ln s,1\}^{-p}$ with $p>1$, and then $\bigcup_{\beta>2\alpha}C^\beta(\bbR^2)\subseteq C_\rho(\bbR^2)$.

Note that since the solutions in Theorem \ref{T1.4} are essentially no better than functions from $L^{1}(\bbR^{2})\cap C^{2\alpha}(\bbR^2)$  at each fixed  time $t$, the corresponding velocities $u(\theta^t)$ are essentially no better than Lipschitz.  Nevertheless, the result implies that the non-linear dynamic still preserves $H^2$-regularity of level sets, at least for a short time, which means that the velocities indeed have a higher degree of regularity when restricted to any of the (infinitely many) individual level sets of the solution.

Moreover, when $\alpha\le \frac 16$ and the initial datum satisfies appropriate hypotheses, we also show that the corresponding solution can be uniquely continued as long as  $H^2$-regularity of the level sets is preserved.  That is, any finite time singularity development must be accompanied by a loss of $H^2$-regularity of solution level sets, and in particular cannot be caused by a collision or ``pile up'' of level sets while these remain regular.  We note that although the question of approach rates and possible collisions of distinct level sets of solutions to active scalar equations has been studied in the past, these works typically considered geometrically somewhat rigid and hence specialized scenarios.  For instance, it was shown in \cite{CorFef} that when at each fixed time  two level sets of a solution to the SQG equation contain graphs of some sufficiently regular functions defined on some time-independent interval in $\bbR\times\{0\}$ and their distance is comparable on that whole interval, then these level sets cannot approach each other faster than double-exponentially.  In contrast, our result allows for arbitrary geometries and relative  positions of $H^2$ level sets for even low regularity solutions to \eqref{1.1}, and excludes not only finite time collisions but also pile ups that could cause solutions to exit $C_\rho(\bbR^2)$ (as long as all level sets remain regular).

The first step in the proof of Theorem \ref{T1.4} is showing that \eqref{1.1}--\eqref{1.2} is locally well-posed in $L^{1}(\bbR^{2})\cap C_\rho(\bbR^2)$ provided the modulus $\rho$ is as above, regardless of the geometry of the level sets of solutions.  This follows from Theorem \ref{T1.2} below, which is then also a marginal improvement of the abovementioned result from \cite{CJK}.  

However, based on our discussion above, one might hope to have local well-posedness for solutions with significantly lower degree of regularity when only their level sets are sufficiently regular.  Nevertheless, this question remains open, the reason for which can be seen from the proofs in our recent work \cite{JeoZla2}.  In it we considered patch solutions for \eqref{1.1}--\eqref{1.2} with possibly touching $H^2$ patches and obtained local well-posedness as long as only exterior touches of patches $\Theta_n^t$ and $\Theta_{n'}^t$ with $\mu_n\mu_{n'}<0$ and interior touches of patches with $\mu_n\mu_{n'}>0$ are allowed.  In particular, one can consider arbitrary initial  configurations of $H^2$ patches satisfying $\Theta_1^0\subseteq \Theta_2^0\subseteq \dots \subseteq \Theta_N^0$  with all $\mu_n>0$, which can be used to approximate  (as $N\to\infty$) general, even discontinuous, functions whose  level sets are simple closed $H^2$ curves.  The problem is, however, that in \cite{JeoZla2} we were only able to control the rate of change of $\sum_{n=1}^N \mu_n\|\partial\Theta^t_n\|_{H^2}^2$ rather than of the individual $H^2$ norms in this sum, which only yields an $N$-dependent  bound on each $H^2$ norm.  This is why we have to stay within $C_\rho(\bbR^2)$  for solutions taking infinitely many values, and in fact the velocities $u(\theta^t)$ of all solutions considered here will be Lipschitz at any fixed time $t$ from their intervals of existence
%, including in directions normal to the level sets of the solutions 
(see Lemma \ref{L2.1} below).

Let us now turn to the level-set view of the dynamic of \eqref{1.1}.
To work directly with level sets of some $\theta\in L^{1}(\bbR^{2})\cap L^{\infty}(\bbR^{2})$, one can use its {\it layer cake decomposition}
\begin{equation}\label{11.4}
    \theta(x) = \int_{0}^{\sup \tht} \mathbbm{1}_{\{\tht> \lambda\}}(x)\,d\lambda - \int_{\inf\tht}^{0} \mathbbm{1}_{\{\tht< \lambda\}}(x)\,d\lambda.
\end{equation}
We will in fact use a generalization of this concept, where the relevant $\lambda$-dependent domains need not be nested and hence need not be sub- and/or super-level sets of $\theta$.  

\begin{definition} \lb{D1.0}
Let $\mathcal{L}$ be a measurable space (with some $\sigma$-algebra), let
$\mu$ be a $\sigma$-finite signed measure on $\mathcal{L}$, and let
$\Theta\subseteq \bbR^{2}\times\mathcal{L}$ be
%some finite measure set 
a set in the product $\sigma$-algebra
(i.e., the $\sigma$-algebra generated by measurable rectangles) of finite measure
%in the product $\sigma$-algebra of $\bbR^{2}\times\mathcal{L}$ 
with respect to the product of the Lebesgue measure on $\bbR^2$ and the total variation measure $|\mu|$ of $\mu$ on $\mathcal{L}$.  For any $\lambda\in\mathcal{L}$, let  $\Theta^{\lambda}:=\{x\in\bbR^2\,:\, (x,\lambda)\in\Theta \}$ be the $\lambda$-section of $\Theta$.  If for each $x\in\bbR^2$ we have
\begin{equation}\label{1.3}
    \theta(x) = \int_{\mathcal{L}}\mathbbm{1}_{\Theta^{\lambda}}(x)\,d\mu(\lambda),
\end{equation}
then  $(\Theta,\mu)$ is
%, or simply $\Theta$,
a \emph{generalized layer cake representation} of $\theta$. 
\end{definition}

{\it Remark.} 1. Since the measurable space $\mathcal{L}$ is implicitly given by $\Theta$ and $\mu$, we will suppress it in  the notation $(\Theta,\mu)$ but will always denote it by $\mathcal{L}$.
\smallskip

2. We require $\Theta$ to be in the product $\sigma$-algebra
(instead of in its completion with respect to the product measure)
in order to ensure measurability of \eqref{1.100} below.
This is general enough to cover the case of layer cake decompositions,
which can be seen by diadically dividing the second component of
$\Theta$ from \eqref{1.101} below.
\smallskip

Hence  \eqref{11.4}  shows that the layer cake decomposition of any $\theta\in L^{1}(\bbR^{2})\cap L^{\infty}(\bbR^{2})$ is also its generalized layer cake representation, with $\mathcal{L}=(\inf\tht,\sup\tht)$, $d\mu(\lambda)=\sgn(\lambda)d\lambda$, and 
\begin{equation}\label{1.101}
    \Theta= \set{(x,\lambda)\in
    \bbR^{2}\times\left[0,\sup\theta\right)
    \colon
    \theta(x) > \lambda}
    \cup
    \set{(x,\lambda)\in
    \bbR^{2}\times\left(\inf\theta,0\right)
    \colon
    \theta(x) < \lambda}.
\end{equation}
Besides \eqref{1.3} being more general than \eqref{11.4}, another reason for our usage of it is  that it can be a more convenient way to represent functions whose level sets have multiple connected components, and in particular are multiple disjoint $H^{2}$ curves.

%Our first result is local well-posedness of \eqref{1.1}--\eqref{1.2}
%within a class of $\theta\in L^{1}(\bbR^{2})\cap L^{\infty}(\bbR^{2})$
%admitting a decomposition of the form
%\begin{equation}\label{1.3}
%    \theta(x) = \int_{\mathcal{L}}\mathbbm{1}_{\Theta^{\lambda}}(x)\,d\mu(\lambda),
%\end{equation}
%where $\mathcal{L}$ is a measurable space (whose $\sigma$-algebra is not explicitly written),
%$\mu$ is a $\sigma$-finite signed measure on $\mathcal{L}$, $\Theta$ is a set in
%the product $\sigma$-algebra of $\bbR^{2}\times\mathcal{L}$
%with finite measure (with respect to the Lebesgue measure and
%the total variation $|\mu|$ of $\mu$),
%and $\Theta^{\lambda}\subseteq\bbR^{2}$ is the $\lambda$-section of $\Theta$
%for each $\lambda\in\mathcal{L}$.
%
%A natural choice of $(\Theta,\mu)$ is that each $\Theta^{\lambda}$
%is a super-level set of $\theta^{+}$ and $\theta^{-}$, and
%$\mu^{+}$, $\mu^{-}$ are respectively the uniform measures on
%$\left(0,\sup\theta\right]$ and
%$\left[\inf\theta,0\right)$, so that \eqref{1.3}
%is the standard layer cake representation.
%In this reason, we call the pair $(\Theta,\mu)$, or simply $\Theta$,
%a \emph{generalized layer cake representation} of $\theta$.\smallskip

%\textit{Remark}. The measurable space $\mathcal{L}$ is implicit from $\Theta$ and $\mu$,
%so we will suppress it in the notation. From now on, $\mathcal{L}$ always
%denotes the measurable space on which $\Theta$ and $\mu$ are defined.\smallskip

Generalized layer cake representations obviously commute with measure-preserving homeomorphisms $\Phi\in C(\bbR^2;\bbR^2)$ in the sense that if $(\Theta,\mu)$ is a
generalized layer cake representation of $\theta$, then $(\Phi_*\Theta,\mu)$ with $\Phi_{*}\Theta\coloneqq\left\{(\Phi(x),\lambda)\colon (x,\lambda)\in\Theta\right\}$  (and so $(\Phi_{*}\Theta)^\lambda = \Phi(\Theta^{\lambda})$) is a
generalized layer cake representation of $\Phi_*\theta:=\theta\circ\Phi^{-1}$.
This, and the fact that our $u(\theta^t)$ will all be Lipschitz, makes the following notion of solutions particularly suited to our approach.

%Given a measure-preserving homeomorphism
%$\Phi\colon\bbR^{2}\to\bbR^{2}$, it can be easily seen that
%$\Phi_{*}\Theta\coloneqq\left\{(\Phi(x),\lambda)\in\bbR^{2}\times\mathcal{L}\colon
%(x,\lambda)\in\Theta\right\}$ is a generalized layer cake representation of
%$\theta\circ\Phi^{-1}$ and
%\[
%    \theta(\Phi^{-1}(x)) = \int_{\mathcal{L}}
%    \mathbbm{1}_{\Phi(\Theta^{\lambda})}(x)\,d\mu(\lambda)
%\]
%for $x\in\bbR^{2}$.

\begin{definition} \lb{D.11.1}
    A \emph{Lagrangian solution}
    to \eqref{1.1}--\eqref{1.2} on an open time interval $I\owns 0$ with
     initial datum $\theta^{0}\in L^{1}(\bbR^{2})\cap L^{\infty}(\bbR^{2})$ is any
    $\theta\colon I\to L^{1}(\bbR^{2})\cap L^{\infty}(\bbR^{2})$ given by
    $\theta^{t}\coloneqq \Phi^{t}_*\theta$ for each $t\in I$, where
    $\Phi\in C\left(I;C(\bbR^{2};\bbR^{2})\right)$ solves the initial value problem
        \begin{equation}\label{1.6}
    %\left\{
    %    \begin{aligned}
           \partial_{t}\Phi^{t} = u(\theta^{t}) \circ \Phi^{t}
     \qquad\text{and}\qquad       \Phi^{0} = {\rm Id}
    %    \end{aligned}
    %\right.
    \end{equation}
    on $I$, with $\Phi^t$ being a measure-preserving homeomorphism for each $t\in I$.
    %for all $x\in\bbR^{2}$
%    where the time derivative is one-sided at any end-point of $I$,
\end{definition}

{\it Remarks.}  1. Here $C(A;B)$ will always refer to the space of all continuous $f:A\to B$, not just the bounded ones.  Although we may then  have $\|f\|_{L^\infty}=\infty$, it still makes sense to consider continuity of $\Phi$ in time with respect to the metric $d(f,g) \coloneqq \norm{f - g}_{L^{\infty}}$ on $C(\bbR^{2};\bbR^{2})$.\smallskip

2.  We will call the above $\Phi$ the \emph{flow map} of $\theta$.  Since $\nabla^\perp$ in \eqref{1.2} implies that $\nabla\cdot u(\theta^{t})\equiv 0$, transport by $u$ preserves the Lebesgue measure and hence so must each $\Phi^t$.
    \smallskip
    
 3. It is easy to show that any Lagrangian solution $\theta$ to \eqref{1.1}--\eqref{1.2} is also
a weak solution in the sense that for all $\varphi\in C_{c}^{1}(\bbR^{2})$ and $t\in I$ we  have
\[
    \frac{d}{dt}\int_{\bbR^{2}}\varphi(x)\,\theta^{t}(x)\,dx
    = \int_{\bbR^{2}}\left[u(\theta^{t};x)\cdot\nabla\varphi(x)\right]
    \theta^{t}(x)\,dx.
\]
\smallskip

To obtain a general local well-posedness result for \eqref{1.1}--\eqref{1.2},
we will also need to add a hypothesis that guarantees $u(\tht^0)$ to be Lipschitz.  In terms of \eqref{1.3}, this will be that
\begin{equation}\label{1.4}
    L_{\mu}(\Theta)\coloneqq \sup_{x\in\bbR^{2}}\int_{\mathcal{L}}
    \frac{d|\mu|(\lambda)}{d(x,\partial\Theta^{\lambda})^{2\alpha}} < \infty
\end{equation}
holds for some generalized layer cake representation $(\Theta,\mu)$ of $\theta^0$ (with $\frac{1}{0}: = \infty$).
% and $0\cdot\infty := 0$).
%where $\abs{\mu}$ is the total variation of $\mu$.
%(This quantity also depends on $\mathcal L$ and $\mu$, but since these will be time-independent for any solution to \eqref{1.1}, we will suppress them in the notation.)
Note that a standard measure theory argument shows joint measurability of 
\begin{equation}\label{1.100}
    d(x,\partial\Theta^{\lambda}) = \max\{
        d(x,\Theta^{\lambda}),
        d(x,\bbR^{2}\setminus\Theta^{\lambda})
    \}
\end{equation}
in $(x,\lambda)$, so $L_{\mu}(\Theta)$ is always well-defined and Lemma \ref{L2.1} below shows that $u(\tht^0)$ is then Lipschitz.
We  note  that   $L_{\mu}(\Phi^t_*\Theta)$ will also be used to control the growth of
$H^{2}$ norms of the boundary curves $\partial((\Phi^t_* \Theta)^{\lambda})= \Phi^t(\partial\Theta^\lambda)$ in the proof of Theorem \ref{T1.4} (see Lemma~\ref{L3.7} below), and that  for any generalized layer cake representation $(\Theta,\mu)$ and any measure-preserving homeomorphism $\Phi\in C(\bbR^2;\bbR^2)$ we have
%which was the motivation for us to study the well-posedness problem
%under the condition \eqref{1.4}.
\begin{equation}\label{1.7}
%    \begin{aligned}
        L_{\mu}(\Phi_{*}\Theta)
%        &= \sup_{x\in\bbR^{2}}\int_{\mathcal{L}}\frac{d|\mu|(\lambda)}
%        {d(x,\partial\Phi(\Theta^{\lambda}))^{2\alpha}}
%        = \sup_{x\in\bbR^{2}}\int_{\mathcal{L}}\frac{d|\mu|(\lambda)}
%        {d(x,\Phi(\partial\Theta^{\lambda}))^{2\alpha}} \\
        = \sup_{x\in\bbR^{2}}\int_{\mathcal{L}}\frac{d|\mu|(\lambda)}
        {d(\Phi(x),\Phi(\partial\Theta^{\lambda}))^{2\alpha}}
        \leq \norm{\Phi^{-1}}_{\dot{C}^{0,1}}^{2\alpha} L_{\mu}(\Theta).
 %   \end{aligned}
    \end{equation}
    
Moreover, when $\theta\in L^{1}(\bbR^{2})\cap L^{\infty}(\bbR^{2})$ has
a modulus of continuity $\rho$ that satisfies \eqref{11.5}
(which includes $\rho(s):=s^{2\alpha} \max\{-\ln s,1\}^{-p}$ with any $p>1$,
when $\bigcup_{\beta>2\alpha}C^\beta(\bbR^2)\subseteq C_\rho(\bbR^2)$),
then \eqref{1.4} also holds when $(\Theta,\mu)$
is given by the layer cake decomposition of $\theta$ from \eqref{11.4}.
%This is because the substitution $s=\rho^{-1}(\lambda)$ and subsequent integration by parts show that
%\beq\lb{11.6}
% L_{\mu}(\Theta)\le 2\int_0^{\sup\tht^0-\inf\tht^0} \rho^{-1}(\lambda)^{-2\alpha} d\lambda
%  = 2\int_0^{\rho^{-1}(\sup\tht^0-\inf\tht^0)} \frac{\rho'(s)}{s^{2\alpha}} ds <\infty.
%\eeq
Indeed, with $\kappa(\lambda)\coloneqq\inf\set{s\geq 0\colon \rho(s)\geq \lambda}$
we have $\kappa^{-1}\left((a,b]\right) = (\rho(a),\rho(b)]$ for any $0\leq a<b$, so a
change of variables and subsequent integration by parts show that
\begin{equation*}
    L_{\mu}(\Theta) \leq 2\int_{0}^{\sup\theta^{0} - \inf\theta^{0}}
    \frac{d\lambda}{\kappa(\lambda)^{2\alpha}}
    = 2\int_{0}^{\kappa\left(\sup\theta^{0} - \inf\theta^{0}\right)}
    \frac{d\kappa_{*}m(s)}{s^{2\alpha}}
    = 2\int_{0}^{\kappa\left(\sup\theta^{0} - \inf\theta^{0}\right)}
    \frac{d\rho(s)}{s^{2\alpha}} < \infty,
\end{equation*}
where $m$ is the 1-dimensional Lebesgue measure and
the last integral is a Lebesgue-Stieltjes integral.
On the other hand, for any $x,y\in\bbR^2$ we have 
    \begin{align*}
        \abs{\theta(x) - \theta(y)}
        \leq \int_{\mathcal{L}}
        \abs{\mathbbm{1}_{\Theta^{\lambda}}(x) - \mathbbm{1}_{\Theta^{\lambda}}(y)}
        \,d|\mu|(\lambda)
        \leq \int_{\mathcal{L}}\frac{\abs{x - y}^{2\alpha}}
        {d(x,\partial\Theta^{\lambda})^{2\alpha}}\,d|\mu|(\lambda)
        \leq L_{\mu}(\Theta)\abs{x - y}^{2\alpha}
    \end{align*}
    because the first integrand is nonzero
    only when $d(x,\partial\Theta^{\lambda})\le \abs{x - y} $, and so $\norm{\theta}_{\dot{C}^{0,2\alpha}} \leq L_{\mu}(\Theta)$ holds for any generalized layer cake representation $(\Theta,\mu)$ of $\theta$.  This all shows that \eqref{1.4} always implies $\theta\in C^{2\alpha}(\bbR^2)$ (with the latter not guaranteeing Lipschitzness of $u(\theta)$) but it is also very close to being equivalent to $\theta\in C^{2\alpha}(\bbR^2)$.  In fact, for any $a>0$,  the layer cake decomposition of  $a\sum_{n\ge 1}  {3^{1-2\alpha n}}  [1-3^n|x-(2^{1-n},0)|]_+$ satisfies \eqref{1.4} but all its moduli of continuity $\rho$ have $\rho(s)\ge as^{2\alpha}$ for all $s\in[0,1]$.

We are now ready to state our general well-posedness result.

\begin{theorem}\label{T1.2}
    Assume that $\theta^{0}\in L^{1}(\bbR^{2})\cap L^{\infty}(\bbR^{2})$ admits
    a generalized layer cake representation $(\Theta,\mu)$ with
    $L_{\mu}(\Theta)<\infty$. Then there is an open interval $I\ni 0$ and a Lagrangian solution
    $\theta$ to \eqref{1.1}--\eqref{1.2}  on $I$    with $\alpha\in(0,\frac 12)$ and initial data $\theta^{0}$, with
    the associated flow map $\Phi\in C\left(I;C(\bbR^{2};\bbR^{2})\right)$ such that
    $\sup_{t\in J}L_{\mu}(\Phi_{*}^{t}\Theta) < \infty$ for any compact interval $J\subseteq I$.  Let $I$ be the maximal such interval.
    Then the solution $\theta$ is unique and independent of the choice of $(\Theta,\mu)$, for any compact interval $J\subseteq I$ we have  $\sup_{t\in J}\max\set{\norm{\Phi^{t}}_{\dot{C}^{0,1}},    \norm{(\Phi^{t})^{-1}}_{\dot{C}^{0,1}}}<\infty$, and for any endpoint $T$  of $I$ we have either $|T|=\infty$  or
    $\lim_{t\to T}L_{\mu}(\Phi_{*}^{t}\Theta) = \infty$.
\end{theorem}

{\it Remark.}  As mentioned above,
%The bound \eqref{11.6} shows that 
the hypothesis is satisfied whenever $\theta^{0}\in L^{1}(\bbR^{2})\cap L^{\infty}(\bbR^{2})$ has modulus of continuity $\rho$ satisfying \eqref{11.5}.
\smallskip

We next turn to our main result, which in particular shows that if, in the setting of Theorem~\ref{T1.2}, each $\partial\Theta^{\lambda}$ is an $H^{2}$ curve
and certain additional assumptions are satisfied, then  the solution from Theorem~\ref{T1.2} retains these properties on some open time interval $I'\ni 0$.
To state it precisely, let $\operatorname{CC}(\bbR^{2})$ be the space of equivalence classes of planar closed curves from $C(\bbT;\bbR^2)$ with respect to the equivalence relation given by the {\it Fr\' echet pseudometric}  
\[
    d_{\mathrm{F}} (\tilde{\gamma}_{1},\tilde{\gamma}_{2})\coloneqq
    \inf_{\phi}    \norm{\tilde{\gamma}_{1} - \tilde{\gamma}_{2}\circ\phi}_{L^{\infty}(\bbT)},
\]
     where the infimum is taken over all orientation-preserving homeomorphisms
$\phi \colon\bbT\to\bbT$ (and $\bbT$ is $[0,1]$ with 0 and 1 identified).  Any $\tilde{\gamma}\in C(\bbT;\bbR^2)$ belonging to some equivalence class $\gamma\in \operatorname{CC}(\bbR^{2})$ (which we also call a {\it closed curve}) will be called a {\it representative} of $\gamma$, and we denote $\operatorname{im}(\gamma):=\operatorname{im}(\tilde\gamma)$ and let the \emph{length} $\ell(\gamma)$ of $\gamma$ be the total variation of $\tilde\gamma$.
 If $\ell(\gamma)<\infty$, then $\gamma$ is rectifiable and its {\it arclength parametrization} is any $\tilde{\gamma}\in C(\ell (\gamma) \bbT;\bbR^{2})$ such that $|\partial_s \tilde\gamma(s)| =1$ for almost all $s\in \ell (\gamma) \bbT$ and $\tilde\gamma(\ell (\gamma) \,\cdot)$ is a representative of $\gamma$.
% , while a {\it constant-speed parametrization} of  $\gamma$ is any parametrization $\tilde{\gamma}\in C(\bbT;\bbR^{2})$ of $\gamma$ satisfying $|\partial_\xi \tilde\gamma(\xi)| =\ell (\gamma)$ for almost all $\xi\in  \bbT$.
  We then also denote by $\norm{{\gamma}}_{\dot{H}^{2}}:= \norm{{\tilde \gamma}}_{\dot{H}^{2}(\ell(\gamma)\bbT)}$ the  $L^2$ norm of the curvature of $\gamma$ (if it is finite, then $\gamma$ is an {\it $H^2$ curve}).
           Finally, we let $\Delta(\gamma_1,\gamma_2):=d(\operatorname{im}(\gamma_1), \operatorname{im}(\gamma_2))$ for any $\gamma_1,\gamma_2\in\operatorname{CC}(\bbR^{2})$ (with $d$ the distance of sets),
%$ \coloneqq    \min_{\xi_1,\xi_2\in\bbT}\abs{\tilde{\gamma}_1(\xi_1) - \tilde{\gamma}_2(\xi_2)}$,
 and let $\operatorname{PSC}(\bbR^{2})\subseteq \operatorname{CC}(\bbR^{2})$ be the set of all positively oriented simple closed curves    in $\bbR^{2}$.  All these definitions appear in \cite{JeoZla2}.
        
     We can now define the quantities relevant to studying solutions with $H^2$ level sets.
        
\begin{definition}
    Let $(\Theta,\mu)$ be a generalized layer cake representation  $\theta\in L^{1}(\bbR^{2})\cap L^{\infty}(\bbR^{2})$.  If for each
    $\lambda\in\mathcal{L}$ there is $z^{\lambda}\in\operatorname{PSC}(\bbR^{2})$  such that $\Theta^{\lambda}$ is a bounded open set with $\partial\Theta^{\lambda} = \operatorname{im}(z^{\lambda})$, then 
    $(\Theta,\mu)$ (or just $\Theta)$ is \emph{composed of simple closed curves}, and we let
    \begin{equation*}
    %\begin{gathered}
        R_{\mu}(\Theta) \coloneqq \sup_{\lambda\in\mathcal{L}}
        \ell(z^{\lambda})^{1/2}
        %\overline{\int_{\mathcal{L}}}
        \int_{\mathcal{L}}
        \frac{d|\mu|(\lambda')}
        {\ell(z^{\lambda'})^{1/2}\Delta(z^{\lambda},z^{\lambda'})^{2\alpha}} \qquad\text{and}\qquad
        Q(\Theta) \coloneqq \sup_{\lambda\in\mathcal{L}}
        \ell(z^{\lambda})\norm{z^{\lambda}}_{\dot{H}^{2}}^{2}.
%        \\
%        \Sigma(\Theta) \coloneqq \inf_{\lambda\in\mathcal{L}}\abs{\Theta^{\lambda}},
%        \qquad\text{and}\qquad
%        \Lambda(\Theta) \coloneqq \sup_{\lambda\in\mathcal{L}}\ell(z^{\lambda}).
%    \end{gathered}
    \end{equation*}
    We also denote $\ell(\partial\Theta^{\lambda}):=\ell(z^{\lambda})$ and $\norm{\partial\Theta^{\lambda}}_{\dot{H}^{2}}:=\norm{z^{\lambda}}_{\dot{H}^{2}}$.
\end{definition}

\smallskip
\textit{Remark}. Note that $\Delta(z^{\lambda},z^{\lambda'}) = \inf_{x\in\bbQ^{2}}\left(
    d(x,\partial\Theta^{\lambda}) + d(x,\partial\Theta^{\lambda'})
\right)$ is measurable in $\lambda'$ since $d(x,\partial\Theta^{\lambda'})$
is jointly measurable in $(x,\lambda')$. Also, a standard measure theory argument
shows that $\lambda'\mapsto\Theta^{\lambda'}$ is measurable with respect to the
topology on the set of all measurable subsets of $\bbR^{2}$ given by the family of pseudometrics
$(A,B)\mapsto |(A\,\triangle\,B)\cap K|$ for all (fixed) compacts $K\subseteq\bbR^{2}$.
Then lower semi-continuity of the perimeter functional (in the sense of Caccioppoli)
with respect to this topology (see \cite[Proposition 3.38]{AFPBV}) shows measurability of $\lambda'\mapsto\ell(z^{\lambda'})$
 (and \cite[Proposition 3.62]{AFPBV} and \cite[Theorem I]{CaffaRivi} show that
$\ell(z^{\lambda'})$ is the perimeter of $\Theta^{\lambda'}$).
Hence $R_{\mu}(\Theta)$ is well-defined.
\smallskip

% Here, $\overline{\int_{\mathcal{L}}}f(\lambda)\,d|\mu|(\lambda)$
% for any function $f\colon\mathcal{L}\to[0,\infty]$ is the \emph{upper Lebesgue integral}
% of $f$; that is,
% \[
%     \overline{\int_{\mathcal{L}}}f(\lambda)\,d|\mu|(\lambda)
%     \coloneqq \inf_{g}\int_{\mathcal{L}}g(\lambda)\,d|\mu|(\lambda)
% \]
% where $g\colon\mathcal{L}\to[0,\infty]$ ranges over all
% measurable functions bounded below by $f$.

We will next add to $L_{\mu}(\Theta)<\infty$ the hypothesis
that $\Theta$ is composed of simple closed curves and
$R_{\mu}(\Theta), Q(\Theta) < \infty$. Assuming $Q(\Theta)<\infty$ ensures
a form of scaling-invariant uniform $H^{2}$ regularity of the curves $z^{\lambda}$
because for any $\gamma\in\operatorname{CC}(\bbR^{2})$ and $a>0$ we have
$\norm{a\gamma}_{\dot{H}^{2}}^{2} = \frac{1}{a}\norm{\gamma}_{\dot{H}^{2}}^{2}$
and $\ell(a\gamma) = a\ell(\gamma)$.
Hypothesis $R_{\mu}(\Theta)<\infty$ is a version of $L_{\mu}(\Theta)<\infty$ but with individual points $x$ replaced by whole curves $z^\lambda$, and it also controls how densely
$z^{\lambda}$ of different scales can be packed together.
More specifically, it prevents too many ``small'' curves $z^{\lambda'}$ to be close to some significantly ``larger'' $z^{\lambda}$. When $(\Theta,\mu)$ is given by the layer cake decomposition of some $\theta$, so the $z^{\lambda}$ are (the  boundaries of) the level sets of $\theta$, this assumption
in effect rules out bumps with ``too sharp'' tops or bottoms.
For instance, if $(\Theta,\mu)$ is the layer cake decomposition of
a radial bump $\theta(x)\coloneqq \left[1 - \abs{x}^{\beta}\right]_{+}$
for some $\beta\in(0,1]$, then $R_{\mu}(\Theta)<\infty$ if and only if
$\beta>\max\set{\frac{1}{2},2\alpha}$.
On the other hand, if instead $\theta(x) \coloneqq \left[1 - \abs{x}\right]_{+}^{\beta}$,
then $R_{\mu}(\Theta)<\infty$ if and only if $\beta>2\alpha$.
Note that in both cases $L_{\mu}(\Theta)<\infty$ if and only if $\beta>2\alpha$.

We can now finally state our  main result. Its part (ii) proves well-posedness for low regularity solutions to \eqref{1.1} with $H^2$ level sets (when all $\partial\Theta^\lambda$ in the generalized layer cake representation considered there are connected components of (boundaries of) level sets of the solution). Moreover, its part (iii) shows that for $\alpha\le\frac 16$ and initial conditions satisfying some additional hypotheses, the corresponding solutions can stop existing only when their level sets lose $H^2$-regularity, and never  just due to level set collisions or pile ups.

\begin{theorem}\label{T1.4}
    Let $\theta^{0}\in L^{1}(\bbR^{2})\cap L^{\infty}(\bbR^{2})$ admit a
    generalized layer cake representation $(\Theta,\mu)$
    composed of simple closed curves such that $L_{\mu}(\Theta),  R_{\mu}(\Theta), Q(\Theta) < \infty$, and let $\theta\colon I\to L^{1}(\bbR^{2})\cap L^{\infty}(\bbR^{2})$ be
    the  Lagrangian solution to \eqref{1.1}--\eqref{1.2} with $\alpha\in(0,\frac 12)$
%    with the initial data $\theta^{0}$ and the associated flow map
%    $\Phi\in C\left(I;C(\bbR^{2};\bbR^{2})\right)$ provided by 
from Theorem~\ref{T1.2}.

(i)    Then $\Phi_{*}^{t}\Theta$ is composed of simple closed curves for each $t\in I$, and
    $\sup_{t\in J}R_{\mu}(\Phi_{*}^{t}\Theta) < \infty$ for each compact interval
    $J\subseteq I$. 
    
    (ii) There is an open interval $I'\subseteq I$ containing 0 such that  $\sup_{t\in J}  Q(\Phi_{*}^{t}\Theta) < \infty$
    for each compact interval    $J\subseteq I'$.  Let $I'$ be the maximal such interval and $T$ its endpoint.
    % that is not an endpoint of $I$, then
    Then either  $|T|=\infty$ or $\lim_{t\to T} L_\mu(\Phi_{*}^{t}\Theta) = \infty$ or
    $\lim_{t\to T} Q(\Phi_{*}^{t}\Theta) = \infty$.
    
  (iii) If also $\alpha\in\left(0,\frac{1}{6}\right]$,
  $|\mu|(\mathcal L)<\infty$, $\int_{\mathcal{L}}\ell(\partial\Theta^\lambda)\,d|\mu|(\lambda)<\infty$, and $\inf_{\lambda\in\mathcal{L}}\abs{\Theta^{\lambda}}>0$,
%$ \sup_{\lambda\in\mathcal{L}}\ell(z^{\lambda})<\infty$, 
 then for
  % If $\alpha\in(0, \frac 16]$ and $T$ is an endpoint of  $I'$,
  any endpoint $T$ of $I'$ from (ii) we have either $|T|=\infty$ or $\limsup_{t\to T} Q(\Phi_{*}^{t}\Theta) = \infty$.
\end{theorem}

{\it Remark.}  Note that the last hypothesis in (iii) excludes initial data that are bumps with rounded tops/bottoms (instead, the latter must be flat), and $\Phi^t$ being measure-preserving then also implies  $\inf_{(t,\lambda)\in I'\times  \mathcal{L}} \operatorname{diam}(\Phi^{t}(\Theta^\lambda))>0$.   When also $\sup_{\lambda\in \mathcal{L}} \operatorname{diam}(\Theta^\lambda)<\infty$, then \eqref{4.2} below shows that $\sup_{(t,\lambda)\in [0,T)\times \mathcal{L}}\operatorname{diam}(\Phi^{t}(\Theta^{\lambda}))<\infty$ when $T$ in (iii) is finite.  So in that case the claim of (iii) implies
\[
\limsup_{t\to T}\, \sup_{\lambda\in \mathcal{L}} \norm{\Phi^{t}(\partial\Theta^{\lambda})}_{\dot{H}^{2}}  = \infty
\]
(while the diameters of $\Phi^{t}(\Theta^{\lambda})$ remain uniformly bounded above and below by positive constants), which translates to true loss of $H^2$-regularity of the solution level sets at time $T$.

\vskip 3mm
\noindent
{\bf Acknowledgement.}  Both authors were supported in part by NSF grant DMS-2407615.

\section{Proof of Theorem~\ref{T1.2}}

All constants $C_{\alpha}$ below
can change from one inequality to another, but they always only depend  on $\alpha$.  We will construct solutions to \eqref{1.1}--\eqref{1.2} as limits of solutions  to a family of similar problems with smooth velocities
\beq \lb{11.3}
    u_{\eps}(\theta^t;x)\coloneqq
    \int_{\bbR^{2}}\nabla^{\perp}K_{\eps}(x-y)\,\theta^t(y)dy,
\eeq
%for $x\in\bbR^{2}$.
where $K_{\eps}(x) \coloneqq \chi(\eps^{-1}{\abs{x}} )K(x)$ for any $\eps>0$ and $x\in\bbR^{2}$, with $\chi\in C^\infty(\bbR)$
is even and satisfying $\mathbbm{1}_{\bbR\setminus(-1,1)} \le \chi\leq \mathbbm{1}_{\bbR\setminus(-1/2,1/2)}$.  Note that for any $n\ge 0$, there is $C_{\alpha,n}$
that only depends on $\alpha,n$ such that the norm of the $n$-linear form
$D^{n}K_{\eps}(x)$ is bounded by
$\frac{C_{\alpha,n}}{\max\{\abs{x},\eps\}^{n+2\alpha}}$.  In particular, this norm is a bounded function of $x\in\bbR^2$ for any $\eps>0$, and we clearly have
\begin{equation}\label{1.8}
    \norm{D^{n}(u_{\eps}(\theta^t))}_{L^{\infty}} \leq
    \norm{D^{n}(\nabla^{\perp}K_{\eps})}_{L^{\infty}}\norm{\theta^t}_{L^1}.
\end{equation}
%for all $n\ge 0$.

Let us start with some estimates on the velocity fields from \eqref{11.7} and \eqref{11.3}
in terms of $L_{\mu}(\Theta)$ for some generalized layer cake representation  $(\Theta,\mu)$  of $\theta$ (in these, we drop $t$ from the notation for convenience).

%Consider some smooth even $\chi\colon\bbR\to\bbR$
%such that $0\leq\chi\leq 1$, $\chi\equiv 1$ on $\bbR\setminus(-1,1)$, and
%$0\notin\operatorname{supp}\chi$. For each $\eps>0$, let
%$K_{\eps}(x) \coloneqq \chi\left(\frac{\abs{x}}{\eps}\right)K(x)$.
%Note that for any $n\ge 0$ there is $C_{\alpha,n}$
%that only depends on $\alpha,n$ such that the norm of the $n$-linear form
%$D^{n}K_{\eps}(x)$ is bounded by
%$\frac{C_{\alpha,n}}{\max\{\abs{x},\eps\}^{n+2\alpha}}$.
%For any finite signed Borel measure $\theta$ on $\bbR^{2}$
%we now define the mollified velocity field
%\beq \lb{11.3}
%    u_{\eps}(\theta^t;x)\coloneqq
%    \int_{\bbR^{2}}\nabla^{\perp}K_{\eps}(x-y)\,\theta^t(y)dy
%\eeq
%for $x\in\bbR^{2}$. Since $\nabla^{\perp}K_{\eps}$ is a smooth function
%whose all derivatives vanish at infinity,
%this integral is always well-defined and $u_{\eps}(\theta)$ is a smooth function such that   ($L^1$, not measures!!!)
%\begin{equation}\label{1.8}
%    \norm{D^{k}(u_{\eps}(\theta))}_{L^{\infty}} \leq
%    \norm{D^{k}(\nabla^{\perp}K_{\eps})}_{L^{\infty}}\norm{\theta}_{\mathrm{TV}}
%\end{equation}
%for all $k\in\bbZ_{\geq 0}$.

\begin{lemma}\label{L2.1}
    There is $C_{\alpha}$ such that for any
    $\theta\in L^{1}(\bbR^{2})\cap L^{\infty}(\bbR^{2})$
    with a generalized layer cake representation $(\Theta,\mu)$, and for any
    $\eps>0$ and $x\in\bbR^{2}$, we have
    \[
        \abs{D(u_{\eps}(\theta))(x)} \leq
        C_{\alpha}\int_{\mathcal{L}}\frac{d|\mu|(\lambda)}
        {d(x,\partial\Theta^{\lambda})^{2\alpha}}.
    \]
    Therefore,
    \[
        \norm{u_\eps(\theta)}_{\dot{C}^{0,1}} \leq C_{\alpha}L_{\mu}(\Theta)
        \qquad\text{and} \qquad \norm{u(\theta)}_{\dot{C}^{0,1}} \leq C_{\alpha}L_{\mu}(\Theta).
    \]
\end{lemma}

\begin{proof}
    For each $\eps>0$ and $x,h\in\bbR^{2}$, oddness of $\nabla^{\perp}K_{\eps}$ shows that
    \begin{align*}
        u_{\eps}(\theta;x+h) - u_{\eps}(\theta;x)
        = \int_{\bbR^{2}}\left(
            \nabla^{\perp}K_{\eps}(x + h - y) - \nabla^{\perp}K_{\eps}(x - y)
        \right)
        (\theta(y) - \theta(x))\,dy.
    \end{align*}
    Replacing $h$ by $sh$ with $s\in\bbR$, and then taking $s\to 0$ yields
    \begin{align*}
        D(u_{\eps}(\theta))(x)h
        &= \int_{\bbR^{2}}
        D(\nabla^{\perp}K_{\eps})(x-y)h \,
        (\theta(y) - \theta(x))\,dy \\
        &= \int_{\bbR^{2}}\int_{\mathcal{L}}
        D(\nabla^{\perp}K_{\eps})(x-y)h
        \left(
            \mathbbm{1}_{\Theta^{\lambda}}(y)
            - \mathbbm{1}_{\Theta^{\lambda}}(x)
        \right)
        d\mu(\lambda)\,dy.
    \end{align*}
%    follows, because the difference between $u_{\eps}(\theta;x+h) - u_{\eps}(\theta;x)$
%    and the right-hand side of the first equality of the above is $O\left(\abs{h}^{2}\right)$.
    Note that $\mathbbm{1}_{\Theta^{\lambda}}(y)
    - \mathbbm{1}_{\Theta^{\lambda}}(x) \neq 0$ implies
    $\abs{x - y} \geq d(x,\partial\Theta^{\lambda})$, so
    \begin{align*}
        \abs{D(u_{\eps}(\theta))(x)}
        &\leq \int_{\mathcal{L}}\int_{\abs{x-y} \geq d(x,\partial\Theta^{\lambda})}
        \frac{C_{\alpha}}{\abs{x - y}^{2+2\alpha}}\,dy\,d|\mu|(\lambda)
        \leq \int_{\mathcal{L}}\frac{C_{\alpha}}
        {d(x,\partial\Theta^{\lambda})^{2\alpha}}d|\mu|(\lambda).
    \end{align*}
   This proves the first and second claims, and the third follows by taking 
   %the supremum over $x\in\bbR^{2}$ and letting 
   $\eps\to 0^{+}$.
\end{proof}

\begin{lemma}\label{L2.2}
    There is $C_{\alpha}$ such that for any
    $\theta\in L^{1}(\bbR^{2})\cap L^{\infty}(\bbR^{2})$ with
    generalized layer cake representations
    $(\Theta_{i},\mu_{i})$ and any measure-preserving homeomorphisms
    $\Phi_{i}\colon\bbR^{2}\to\bbR^{2}$ ($i=1,2$),
    \[
        \norm{u(\Phi_{1*}\theta) - u(\Phi_{2*}\theta)}_{L^{\infty}} \leq
        C_{\alpha} \left(
            L_{\mu_{1}}(\Phi_{1*}\Theta_{1}) + L_{\mu_{2}}(\Phi_{2*}\Theta_{2})
        \right)
        \norm{\Phi_{1} - \Phi_{2}}_{L^{\infty}}.
    \]
\end{lemma}

\begin{proof}
    Let    $\mathcal{L}_{i}$ the measurable space associated to $(\Theta_{i},\mu_{i})$ and
    $\theta_{i}\coloneqq \theta\circ\Phi_{i}^{-1}$ for $i=1,2$.  Let $d\coloneqq\norm{\Phi_{1} - \Phi_{2}}_{L^{\infty}}$ and fix any $x\in\bbR^{2}$.
    Then $u(\Phi_{1*}\theta;x) - u(\Phi_{2*}\theta;x)$ is the sum of
    \begin{align*}
        I_{1} &\coloneqq \int_{\abs{x - y}\leq 2d}
        \nabla^{\perp}K(x - y)(\theta_{1}(y) - \theta_{2}(y))\,dy, \\
        I_{2} &\coloneqq \int_{\abs{x - y} > 2d}
        \nabla^{\perp}K(x - y)(\theta_{1}(y) - \theta_{2}(y))\,dy.
    \end{align*}

    \textbf{Estimate for $I_{1}$.} By oddness of $\nabla^{\perp}K$, we have
    \[
        I_{1} = \int_{\abs{x - y}\leq 2d}\nabla^{\perp}K(x - y)
        (\theta_{1}(y) - \theta_{1}(x))\,dy
        - \int_{\abs{x - y}\leq 2d}\nabla^{\perp}K(x - y)
        (\theta_{2}(y) - \theta_{2}(x))\,dy.
    \]
Since
%    and by symmetry, it is enough to estimate
    \[
        I_{3} \coloneqq \int_{\abs{x - y} \leq 2d}
        \frac{\abs{\theta_{1}(y) - \theta_{1}(x)}}{\abs{x - y}^{1 + 2\alpha}}\,dy
        \leq \int_{\mathcal{L}_{1}}\int_{\abs{x - y}\leq 2d}
        \frac{\abs{\mathbbm{1}_{\Phi_{1}(\Theta_{1}^{\lambda})}(y)
        - \mathbbm{1}_{\Phi_{1}(\Theta_{1}^{\lambda})}(x)}}
        {\abs{x - y}^{1+2\alpha}}\,dy\,d|\mu_{1}|(\lambda)
    \]
    and (as in the proof of Lemma~\ref{L2.1})
    we have $\abs{x - y} \geq d(x,\partial\Phi_{1}(\Theta_{1}^{\lambda}))$
    whenever the  last integrand  is nonzero,
    we see that
    \begin{align*}
        I_{3} &\leq \int_{\mathcal{L}_{1}}\int_{\abs{x - y}\leq 2d}
        \frac{1}{\abs{x - y}\, d(x,\partial\Phi_{1}(\Theta_{1}^{\lambda}))^{2\alpha}}
        \,dy\,d|\mu_{1}|(\lambda)
        \leq 4\pi L_{\mu_{1}}(\Phi_{1*}\Theta_{1})d.
    \end{align*}
    The same argument for $\theta_{2}$ in place of $\theta_{1}$ now yields
    \begin{equation} \label{2.100}
        \abs{I_{1}}\leq C_{\alpha}\left(
            L_{\mu_{1}}(\Phi_{1*}\Theta_{1}) + L_{\mu_{2}}(\Phi_{2*}\Theta_{2})
        \right)d.
    \end{equation}

    \textbf{Estimate for $I_{2}$.} For each $R>2d$ let
    \[
        I_{2}^{R}\coloneqq \int_{2d<\abs{x-y}\leq R}
        \nabla^{\perp}K(x-y)(\theta_{1}(y) - \theta_{2}(y))\,dy,
    \]
    so that $I_{2} = \lim_{R\to\infty}I_{2}^{R}$. Fix $R>2d$, and then $\Phi_i$ being measure-preserving yields
    \begin{align*}
        I_{2}^{R} &= \int_{2d< \abs{x - y} \leq R}
        \nabla^{\perp}K(x - y)(\theta_{1}(y) - \theta_{1}(x))\,dy
        \\&\quad\quad\quad\quad\quad
        - \int_{2d < \abs{x - y} \leq R}
        \nabla^{\perp}K(x - y)(\theta_{2}(y) - \theta_{1}(x))\,dy \\
        &= \int_{2d < \abs{x - \Phi_{1}(y)} \leq R}\nabla^{\perp}K(x - \Phi_{1}(y))
        (\theta(y) - \theta_{1}(x))\,dy
        \\&\quad\quad\quad\quad\quad
        - \int_{2d < \abs{x - \Phi_{2}(y)} \leq R}\nabla^{\perp}K(x - \Phi_{2}(y))
        (\theta(y) - \theta_{1}(x))\,dy \\
        &= \int_{\abs{x - \Phi_{1}(y)},\abs{x - \Phi_{2}(y)}\in (2d,R]}
        \left[
            \nabla^{\perp}K(x - \Phi_{1}(y)) - \nabla^{\perp}K(x - \Phi_{2}(y))
        \right]
        (\theta(y) - \theta_{1}(x))\,dy
        \\&\quad\quad\quad\quad\quad
        + \int_{\abs{x - \Phi_{2}(y)} \leq 2d < \abs{x - \Phi_{1}(y)} \leq R}
        \nabla^{\perp}K(x - \Phi_{1}(y))(\theta(y) - \theta_{1}(x))\,dy
        \\&\quad\quad\quad\quad\quad
        - \int_{\abs{x - \Phi_{1}(y)} \leq 2d < \abs{x - \Phi_{2}(y)} \leq R}
        \nabla^{\perp}K(x - \Phi_{2}(y))(\theta(y) - \theta_{1}(x))\,dy
        \\&\quad\quad\quad\quad\quad
        + \int_{2d < \abs{x - \Phi_{1}(y)} \leq R < \abs{x - \Phi_{2}(y)}}
        \nabla^{\perp}K(x - \Phi_{1}(y))(\theta(y) - \theta_{1}(x))\,dy
        \\&\quad\quad\quad\quad\quad
        - \int_{2d < \abs{x - \Phi_{2}(y)} \leq R < \abs{x - \Phi_{1}(y)}}
        \nabla^{\perp}K(x - \Phi_{2}(y))(\theta(y) - \theta_{1}(x))\,dy.
    \end{align*}
    Let us denote the integrals on the right-hand side  $I_{4},I_{5},I_{6},I_{7},I_{8}$
     (in the order of appearance).

    To estimate $I_{4}$, note that for any $y$ in the domain of integration we have
    \[
        \min_{\eta\in[0,1]}
        \abs{x - (1-\eta)\Phi_{1}(y) - \eta\Phi_{2}(y)}
        \geq \abs{x - \Phi_{1}(y)} - d
        \geq \frac{\abs{x - \Phi_{1}(y)}}{2},
    \]
     so the mean value theorem shows that
    \[
        \abs{\nabla^{\perp}K(x - \Phi_{1}(y)) - \nabla^{\perp}K(x - \Phi_{2}(y))}
        \leq \frac{C_{\alpha}d}{\abs{x - \Phi_{1}(y)}^{2 + 2\alpha}}.
    \]
    The change of variables formula now yields
    \begin{align*}
        \abs{I_{4}} &\leq \int_{\abs{x - y} > 2d}
        \frac{C_{\alpha}d\abs{\theta_{1}(y) - \theta_{1}(x)}}
        {\abs{x - y}^{2 + 2\alpha}}\,dy \\
        &\leq \int_{\mathcal{L}_{1}}\int_{\abs{x - y} > 2d}
        \frac{C_{\alpha}d\abs{\mathbbm{1}_{\Phi_{1}(\Theta_{1}^{\lambda})}(y)
        - \mathbbm{1}_{\Phi_{1}(\Theta_{1}^{\lambda})}(x)}}
        {\abs{x - y}^{2 + 2\alpha}}\,dy\,d|\mu_{1}|(\lambda).
    \end{align*}
    Again, $\abs{x - y} \geq d(x,\partial\Phi_{1}(\Theta_{1}^{\lambda}))$
    holds whenever the last integrand  is nonzero, so
    \begin{align*}
        \abs{I_{4}} &\leq \int_{\mathcal{L}_{1}}
        \int_{\abs{x - y}\geq d(x,\partial\Phi_{1}(\Theta_{1}^{\lambda}))}
        \frac{C_{\alpha}d}{\abs{x - y}^{2+2\alpha}}\,dy\,d|\mu_{1}|(\lambda)
        \leq C_{\alpha}L_{\mu_{1}}(\Phi_{1*}\Theta_{1})d.
    \end{align*}

    For $I_{5}$, note that for any $y$ in the domain of integration we have
    \[
        \abs{x - \Phi_{1}(y)} \leq \abs{x - \Phi_{2}(y)} + d \leq 3d.
    \]
    By applying again  change of variables we obtain
    \[
        \abs{I_{5}} \leq \int_{\abs{x - y} \leq 3d}
        \frac{C_{\alpha}\abs{\theta_{1}(y) - \theta_{1}(x)}}{\abs{x - y}^{1 + 2\alpha}}\,dy,
    \]
    so the same argument as in the estimate for $I_{3}$
    shows $\abs{I_{5}} \leq C_{\alpha}L_{\mu_{1}}(\Phi_{1*}\Theta_{1})d$.
And clearly
    \[
        \abs{I_{6}} \leq \int_{\abs{x - \Phi_{1}(y)} \leq 2d}
        \frac{C_{\alpha}\abs{\theta(y) - \theta_{1}(x)}}
        {\abs{x - \Phi_{1}(y)}^{1+2\alpha}}\,dy
        = \int_{\abs{x - y}\leq 2d}\frac{C_{\alpha}\abs{\theta_{1}(y) - \theta_{1}(x)}}
        {\abs{x - y}^{1+2\alpha}}\,dy,
    \]
    so again $\abs{I_{6}} \leq C_{\alpha}L_{\mu_{1}}(\Phi_{1*}\Theta_{1})d$.
    % in the same way    as in the estimate of $I_{3}$.

    To estimate $I_{7}$, note that for any $y$ in the domain of integration we have
    \begin{align*}
        \abs{x - \Phi_{1}(y)} \geq \abs{x - \Phi_{2}(y)} - d > R - d,
    \end{align*}
    so the change of variables formula yields
    \begin{align*}
        \abs{I_{7}} \leq \int_{R - d < \abs{x - y} \leq R}
        \frac{C_{\alpha}\norm{\theta}_{L^{\infty}}}
        {\abs{x - y}^{1 + 2\alpha}}\,dy
       % = 2\pi C_{\alpha}\norm{\theta}_{L^{\infty}}
       % R^{1-2\alpha}\int_{1 - \frac{d}{R}}^{1}\frac{dr}{r^{2\alpha}}
        \leq \frac{ C_{\alpha}\norm{\theta}_{L^{\infty}}d}{R^{2\alpha}}
    \end{align*}
    %where the last inequality is 
    because ${R} >2d$.
    In the same way we also obtain $|I_{8}|\le \frac{C_{\alpha}\norm{\theta}_{L^{\infty}}d}{R^{2\alpha}}$.
    
    Collecting the above estimates and letting $R\to\infty$, we see that
    $\abs{I_{2}} \leq C_{\alpha}L_{\mu_{1}}(\Phi_{1*}\Theta_{1})d$.
    This and \eqref{2.100} now hold uniformly in $x\in\bbR^{2}$, finishing the proof.
\end{proof}

\begin{lemma}\label{L2.3}
    There is $C_{\alpha}$ such that for any
    $\theta\in L^{1}(\bbR^{2})\cap L^{\infty}(\bbR^{2})$ and $\eps>0$ we have
    \[
        \norm{u(\theta) - u_{\eps}(\theta)}_{L^{\infty}}
        \leq C_{\alpha}\norm{\theta}_{L^{\infty}}\eps^{1-2\alpha}.
    \]
\end{lemma}

\begin{proof}
    Since $\nabla^{\perp}K_{\eps}(x) = \nabla^{\perp}K(x)$
    when $\abs{x}\geq\eps$, for any $x\in\bbR^{2}$ we have
    \begin{align*}
        \abs{u(\theta;x) - u_{\eps}(\theta;x)}
        &\leq \int_{\abs{x - y}\leq \eps}
        \abs{\nabla^{\perp}K(x-y) - \nabla^{\perp}K_{\eps}(x-y)}
        \norm{\theta}_{L^{\infty}}\,dy \\
        &\leq \int_{\abs{x - y}\leq \eps}
        \frac{C_{\alpha}\norm{\theta}_{L^{\infty}}}{\abs{x - y}^{1 + 2\alpha}}\,dy
        = C_{\alpha}\norm{\theta}_{L^{\infty}}\eps^{1-2\alpha}.
    \end{align*}
\end{proof}

Now, fix any initial datum $\theta^{0}\in L^{1}(\bbR^{2})\cap L^{\infty}(\bbR^{2})$
admitting a generalized layer cake representation
$(\Theta,\mu)$ with $L_{\mu}(\Theta) < \infty$.
Fix any $\eps>0$ and consider the ODE
\begin{equation}\label{2.1}
 %   \left\{
 %       \begin{aligned}
            \partial_{t}\Psi_{\eps}^{t}
            = u_{\eps}((\mathrm{Id} + \Psi_{\eps}^{t})_{*}\theta^{0}) \circ ({\rm Id} + \Psi_{\eps}^{t})
            \qquad\text{and}\qquad
            \Psi_{\eps}^{0} \equiv 0
%        \end{aligned}
%    \right
\end{equation}
with $\Psi_\eps^t \in BC(\bbR^{2};\bbR^{2})$ (the space of bounded continuous functions from $\bbR^{2}$ to $\bbR^{2}$).
That is, $\Phi_{\eps}^{t}\coloneqq \mathrm{Id} + \Psi_{\eps}^{t}$
is the flow map generated by the vector field
$u_{\eps}(\Phi_{\eps *}^{t}\theta^{0})$.
% that is in turn generated by  $\Phi_{\eps *}^{t}\theta^{0}$.  
However, until we show that $\Phi_{\eps}^{t}$ is a measure-preserving homeomorphism, $\Phi_{\eps *}^{t}\theta^{0}$ will be the pushforward of the measure $\theta^0(y)dy$ by $\Phi_{\eps}^{t}$ (which is $(\theta^0\circ(\Phi_\eps^t)^{-1})(y)dy$ when $\Phi_{\eps}^{t}$ is a measure-preserving homeomorphism) and we replace \eqref{11.3} by
\[
    u_\eps(\nu;x) \coloneqq \int_{\bbR^{2}}   \nabla^{\perp}K_\eps(x - y)\,d\nu(y)
\]
for a finite signed Borel measure $\nu$ on $\bbR^{2}$.
% (when the integral converges absolutely).
We will next show that given any $\tht^0$ as above, \eqref{2.1} is globally well-posed in $BC(\bbR^{2};\bbR^{2})$.

% (The reason for considering this general case
%is merely because of convenience in certain aspects of developing the theory, and we will
%not be concerned with the well-posedness of \eqref{1.1}--\eqref{1.2} in this general setting.)
%When $\theta$ is an $L^{1}$ function, we identify it with the finite signed Borel measure
%it defines through integration with respect to the Lebesgue measure, so that
%the interpretations of $u(\theta)$ in both ways are consistent.
%Note that for any measure-preserving homeomorphism $\Phi\colon\bbR^{2}\to\bbR^{2}$,
%this correspondence between $L^{1}$ functions and finite signed Borel measures
%identifies the function $\theta\circ\Phi^{-1}$ with the pushforward measure $\Phi_{*}\theta$.

\begin{lemma}\label{L2.4}
    For each $F\in BC(\bbR^{2};\bbR^{2})$, let
    \[
        \mathcal{F}(F) \coloneqq u_{\eps}((\mathrm{Id} + F)_{*}\theta^{0})\circ ({\rm Id} + F).
    \]
    Then $\mathcal{F}\colon    BC(\bbR^{2};\bbR^{2})\to BC(\bbR^{2};\bbR^{2})$
    is well-defined and  Lipschitz continuous.
\end{lemma}

\begin{proof}
    Clearly $\mathcal{F}(F)\in C(\bbR^{2}; \bbR^{2})$
    for any $F\in BC(\bbR^{2};\bbR^{2})$.
    For any $F_{1},F_{2}\in BC(\bbR^{2};\bbR^{2})$ and $x\in\bbR^{2}$ we see that $(\mathcal{F}(F_{1}) - \mathcal{F}(F_{2}))(x)$ equals
    \begin{equation*}
        \begin{aligned}
&\             \int_{\bbR^{2}} \nabla^{\perp}K_{\eps}(x + F_{1}(x) - y)
            \,d(\mathrm{Id} + F_{1})_{*}\theta^{0}(y)
            - \int_{\bbR^{2}} \nabla^{\perp}K_{\eps}(x + F_{2}(x) - y)
            \,d(\mathrm{Id} + F_{2})_{*}\theta^{0}(y)
            \\&\ 
            = \int_{\bbR^{2}}
            \left[
                \nabla^{\perp}K_{\eps}(x - y + F_{1}(x) - F_{1}(y))
                - \nabla^{\perp}K_{\eps}(x - y + F_{2}(x) - F_{2}(y))
            \right]
            \theta^{0}(y)\,dy,
        \end{aligned}
    \end{equation*}
    so
    \[
        \norm{\mathcal{F}(F_{1}) - \mathcal{F}(F_{2})}_{L^{\infty}}
        \leq 2\norm{D(\nabla^{\perp}K_{\eps})}_{L^{\infty}}\norm{\theta^{0}}_{L^{1}}
        \norm{F_{1} - F_{2}}_{L^{\infty}}.
    \]
    Since $\mathcal{F}(0) = u_{\eps}(\theta^{0})$ is bounded,
    both claims follow from this.
\end{proof}

Lemma~\ref{L2.4} shows that \eqref{2.1} is globally well-posed, and we  let
$\Phi_{\eps}^{t}\coloneqq \mathrm{Id} + \Psi_{\eps}^{t}$ and
$\theta_{\eps}^{t}\coloneqq\Phi_{\eps*}^{t}\theta^{0}$, with the latter being for now only a finite signed Borel measure.
We will next show that $\Phi_{\eps}^{t}$ is in fact a measure-preserving homeomorphism, which will mean that $\theta_{\eps}^{t}=\theta^0\circ(\Phi_\eps^t)^{-1}$ and  it is also an $L^{1}\cap L^\infty$ function.

Clearly the ODE
\begin{equation}\label{2.2}
    \partial_{t}G^{t} = u_{\eps}(\theta_{\eps}^{t})\circ ({\rm Id} + G^{t})
\end{equation}
is globally well-posed in $BC(\bbR^{2};\bbR^{2})$ for any initial data at any initial time.
For each $t_{0},t_1\in\bbR$, let $\Gamma_{\eps}^{t_{0},t}$ be the unique solution to \eqref{2.2}
with initial data $\Gamma_{\eps}^{t_{0},t_{0}}: = 0$ at time $t=t_0$, and consider
$G^{t} \coloneqq \Gamma_{\eps}^{t_{0},t_{1}} +
\Gamma_{\eps}^{t_{1},t}\circ(\mathrm{Id} + \Gamma_{\eps}^{t_{0},t_{1}})$.  Then
 $G^t$ solves \eqref{2.2} and 
$G^{t_{1}} = \Gamma_{\eps}^{t_{0},t_{1}}$, so uniqueness of the solution
with the initial data $\Gamma_{\eps}^{t_{0},t_{1}}$ at time $t=t_{1}$ shows that
\[
    \mathrm{Id} + \Gamma_{\eps}^{t_{0},t}
    = \mathrm{Id} + G^{t}
    = (\mathrm{Id} + \Gamma_{\eps}^{t_{1},t})
    \circ(\mathrm{Id} + \Gamma_{\eps}^{t_{0},t_{1}})
\]
 for all $t\in\bbR$. Letting $t:=t_{0}$ shows that
$(\mathrm{Id} + \Gamma_{\eps}^{t_{1},t_{0}})
\circ(\mathrm{Id} + \Gamma_{\eps}^{t_{0},t_{1}}) = \mathrm{Id}$,
so we conclude that each $\mathrm{Id} + \Gamma_{\eps}^{t_{0},t}$ is a homeomorphism.
Then so is $\Phi_{\eps}^{t} = \mathrm{Id} + \Gamma_{\eps}^{0,t}$.
%and so $\theta_{\eps}^{t}\in L^1(\bbR^2)$.

Letting $BC^{1}(\bbR^{2};\bbR^{2})$ be the space of bounded $C^{1}$ functions from $\bbR^{2}$ to $\bbR^{2}$ with bounded first derivatives, we see that for any $F\in BC^{1}(\bbR^{2};\bbR^{2})$ and $x,h\in\bbR^{2}$ we have
\begin{align*}
    D(u_{\eps}(\theta_{\eps}^{t})\circ(\mathrm{Id} + F))(x)h
    = \int_{\bbR^{2}}D(\nabla^{\perp}K_{\eps})(x + F(x) - y)(h + DF(x)h)
    \,d\theta_{\eps}^{t}(y).
\end{align*}
Therefore
$F\mapsto u_{\eps}(\theta_{\eps}^{t})\circ(\mathrm{Id} + F)$
is locally Lipschitz on $BC^{1}(\bbR^{2};\bbR^{2})$, so
\eqref{2.2} is locally well-posed there.
But since for any $F\in BC^{1}(\bbR^{2};\bbR^{2})$ we have
\[
    \norm{D(u_{\eps}(\theta_{\eps}^{t})\circ(\mathrm{Id} + F))}_{L^{\infty}}
    \leq \norm{D(\nabla^{\perp}K_{\eps})}_{L^{\infty}}\norm{\theta^{0}}_{L^{1}}
    \norm{\mathrm{Id} + DF}_{L^{\infty}},
\]
a Gr\"{o}nwall-type argument shows that the $C^{1}$ norm of any solution to \eqref{2.2}
can grow no faster than exponentially.  Therefore \eqref{2.2} is even globally well-posed in
$BC^{1}(\bbR^{2};\bbR^{2})$, and so $\Gamma_{\eps}^{t_{0},t}\in BC^{1}(\bbR^{2};\bbR^{2})$
for all $t\in\bbR$.  This and $\nabla \cdot u_{\eps}(\theta_{\eps}^{t}) \equiv 0$
now show that the map $\mathrm{Id} + \Gamma_{\eps}^{t_{0},t}$ is measure-preserving.
Then $\theta_{\eps}^{t}=\theta^{0}\circ(\Phi_{\eps}^{t})^{-1}\in L^1(\bbR^2)\cap L^\infty(\bbR^2)$
and  $\Phi_{\eps*}^{t}\Theta$ is its generalized layer cake representation.
%\smallskip

%\textit{Remark}. 
Similarly, with $BC^{2}(\bbR^{2};\bbR^{2})$ the space of
bounded $C^{2}$ functions from $\bbR^{2}$ to $\bbR^{2}$ with bounded first and second derivatives,
for each $F\in BC^{2}(\bbR^{2};\bbR^{2})$ and $x,h_{1},h_{2}\in\bbR^{2}$ we have
\begin{align*}
    &D^{2}(u_{\eps}(\theta_{\eps}^{t})\circ(\mathrm{Id} + F))(x)(h_{1},h_{2})
    \\&\quad\quad = \int_{\bbR^{2}}D^{2}(\nabla^{\perp}K_{\eps})
    (x + F(x) - y)(h_{1} + DF(x)h_{1}, h_{2} + DF(x)h_{2})\,\theta_{\eps}^{t}(y)\,dy
    \\&\quad\quad\quad\quad\quad
    + \int_{\bbR^{2}}D(\nabla^{\perp}K_{\eps})(x + F(x) - y)
    \left(D^{2}F(x)(h_{1}, h_{2})\right)\theta_{\eps}^{t}(y)\,dy
\end{align*}
and
\begin{align*}
    \norm{D^{2}(u_{\eps}(\theta_{\eps}^{t})\circ(\mathrm{Id} + F))}_{L^{\infty}}
    &\leq \norm{D^{2}(\nabla^{\perp}K_{\eps})}_{L^{\infty}} \norm{\theta^{0}}_{L^{1}}
    \norm{\mathrm{Id} + DF}_{L^{\infty}}^{2}
    \\&\quad\quad\quad
    + \norm{D(\nabla^{\perp}K_{\eps})}_{L^{\infty}} \norm{\theta^{0}}_{L^{1}}
    \norm{D^{2}F}_{L^{\infty}}.
\end{align*}
Another Gr\"{o}nwall-type argument and the time-exponential bound on the $C^{1}$ norms of solutions to \eqref{2.2} 
now shows that \eqref{2.2} is globally well-posed in $BC^{2}(\bbR^{2};\bbR^{2})$, which we will use  in Section~\ref{S3}.
One can continue and inductively show that
\eqref{2.2} is globally well-posed in $BC^{k}(\bbR^{2};\bbR^{2})$ for all $k\in\bbN$ (then each $\Phi_{\eps}^{t}$ is a diffeomorphism), but we will not need this here.
%\smallskip

Next we derive an $\eps$-independent estimate on the growth of
$L_{\mu}(\Phi_{\eps*}^{t}\Theta)$.

\begin{lemma}\label{L2.5}
    $\norm{D^{k}(u_{\eps}(\theta_{\eps}^{t}))}_{L^{\infty}}$ is continuous in $t$
    for all $k\in\bbZ_{\geq 0}$, and
    \begin{align}
        &\abs{\Gamma_{\eps}^{t_{0},t_{1}}(x) - \Gamma_{\eps}^{t_{0},t_{1}}(y)
        - \Gamma_{\eps}^{t_{0},t_{2}}(x) + \Gamma_{\eps}^{t_{0},t_{2}}(y)}  \notag
        \\&\qquad\qquad
        \leq \left(\exp\left(\abs{
            \int_{t_{2}}^{t_{1}}\norm{u_{\eps}(\theta_{\eps}^{\tau})}_{\dot{C}^{0,1}}d\tau
        }\right) - 1\right)
        \abs{x + \Gamma_{\eps}^{t_{0},t_{2}}(x) - y - \Gamma_{\eps}^{t_{0},t_{2}}(y)}  \lb{2.101}
    \end{align}
    holds for all $x,y\in\bbR^{2}$ and $t_{0},t_{1},t_{2}\in\bbR$.
\end{lemma}

\begin{proof}
    Change of variables  yields
    \[
        \norm{D^{k}(u_{\eps}(\theta_{\eps}^{t_{1}}))
        - D^{k}(u_{\eps}(\theta_{\eps}^{t_{2}}))}_{L^{\infty}}
        \leq \norm{D^{k+1}(\nabla^{\perp}K_{\eps})}_{L^{\infty}}
        \norm{\theta^{0}}_{L^{1}}
        \norm{\Phi_{\eps}^{t_{1}} - \Phi_{\eps}^{t_{2}}}_{L^{\infty}}
    \]
    for any $(k,t_{1},t_{2})\in\bbZ_{\geq 0}\times\bbR^{2}$.  This shows
    the first claim, and in particular that $\norm{u_{\eps}(\theta_{\eps}^{t})}_{\dot{C}^{0,1}}$
    is continuous in $t$. 

    Next, letting $x':=x + \Gamma_{\eps}^{t_{0},t_{2}}(x)$, we see that
    \[
        \Gamma_{\eps}^{t_{0},t_{1}}(x) = x' + \Gamma_{\eps}^{t_{2},t_{1}}(x') - x
        = \Gamma_{\eps}^{t_{2},t_{1}}(x') + \Gamma_{\eps}^{t_{0},t_{2}}(x).
    \]
    So with $y':=y + \Gamma_{\eps}^{t_{0},t_{2}}(y)$, the left-hand side of \eqref{2.101}
    is just $|\Gamma_{\eps}^{t_{2},t_{1}}(x')-\Gamma_{\eps}^{t_{2},t_{1}}(y')|$,
    while the last factor is $|x'-y'|$. The result now follows from the definition of
    $\Gamma_{\eps}^{t_{2},t}$.
\end{proof}

Next, recall that the \emph{upper-right} and  \emph{lower-left Dini derivatives} of $f:(a,b)\to\bbR$ are
\[
    \partial_{t}^{+}f(t)\coloneqq \limsup_{h\to 0^{+}}\frac{f(t+h) - f(t)}{h} \qquad\text{and}\qquad
    \partial_{t-}f(t)\coloneqq \liminf_{h\to 0^{-}}\frac{f(t+h) - f(t)}{h}.
\]

\begin{proposition}\label{P2.6}
    $L_{\mu}(\Phi_{\eps*}^{t}\Theta)$ is continuous in $t$ and
    \[
        \max\set{
            \partial_{t}^{+}L_{\mu}(\Phi_{\eps*}^{t}\Theta),
            -\partial_{t-}L_{\mu}(\Phi_{\eps*}^{t}\Theta)
        }
        \leq \norm{u_{\eps}(\theta_{\eps}^{t})}_{\dot{C}^{0,1}}
        L_{\mu}(\Phi_{\eps*}^{t}\Theta).
    \]
\end{proposition}

\begin{proof}
    Fix any $t_{1},t_{2}\in\bbR$, $x\in\bbR^{2}$, $\lambda\in\mathcal{L}$, and $\eta>0$.
    Pick $y\in\partial\Theta^{\lambda}$ such that
    \[
      \abs{\Phi_{\eps}^{t_{1}}(x) - \Phi_{\eps}^{t_{1}}(y)} \le   
      d(\Phi_{\eps}^{t_{1}}(x), \partial\Phi_{\eps}^{t_{1}}(\Theta^{\lambda})) + \eta.
    \]
    Then Lemma~\ref{L2.5} and
    the inequality $\abs{\frac{1}{a^{2\alpha}} - \frac{1}{b^{2\alpha}}}
    \leq \frac{\abs{a - b}}{ab^{2\alpha}}$ for $a,b>0$ show that
    \begin{equation}\label{2.4}
        \begin{aligned}
            &\frac{1}{\left(
                d(\Phi_{\eps}^{t_{1}}(x),
                \partial\Phi_{\eps}^{t_{1}}(\Theta^{\lambda})) + 2\eta
            \right)^{2\alpha}}
            - \frac{1}{\left(
                d(\Phi_{\eps}^{t_{2}}(x),
                \partial\Phi_{\eps}^{t_{2}}(\Theta^{\lambda})) + \eta
            \right)^{2\alpha}}
            \\&\quad\quad\quad\quad
            \leq \frac{1}{\left(
                \abs{\Phi_{\eps}^{t_{1}}(x) - \Phi_{\eps}^{t_{1}}(y)} + \eta
            \right)^{2\alpha}}
            - \frac{1}{\left(
                \abs{\Phi_{\eps}^{t_{2}}(x) - \Phi_{\eps}^{t_{2}}(y)} + \eta
            \right)^{2\alpha}}
            \\&\quad\quad\quad\quad
            \leq \frac{\abs{\Phi_{\eps}^{t_{1}}(x) - \Phi_{\eps}^{t_{1}}(y)
            - \Phi_{\eps}^{t_{2}}(x) + \Phi_{\eps}^{t_{2}}(y)}}
            {\left(\abs{\Phi_{\eps}^{t_{1}}(x) - \Phi_{\eps}^{t_{1}}(y)} + \eta\right)
            \left(\abs{\Phi_{\eps}^{t_{2}}(x) - \Phi_{\eps}^{t_{2}}(y)} + \eta\right)^{2\alpha}}
            \\&\quad\quad\quad\quad
            \leq \frac{\exp\left(\abs{
                \int_{t_{2}}^{t_{1}}
                \norm{u_{\eps}(\theta_{\eps}^{\tau})}_{\dot{C}^{0,1}}\,d\tau
            }\right) - 1}
            {(d(\Phi_{\eps}^{t_{2}}(x),
            \partial\Phi_{\eps}^{t_{2}}(\Theta^{\lambda})) + \eta)^{2\alpha}},
        \end{aligned}
    \end{equation}
    so letting $\eta\to 0^{+}$, integrating over $\lambda$,
    and then taking supremum over $x\in\bbR^2$ shows
    \[
        L_{\mu}(\Phi_{\eps*}^{t_{1}}\Theta)
        \leq \exp\left(\abs{
            \int_{t_{2}}^{t_{1}}
            \norm{u_{\eps}(\theta_{\eps}^{\tau})}_{\dot{C}^{0,1}}\,d\tau
        }\right)L_{\mu}(\Phi_{\eps*}^{t_{2}}\Theta).
    \]
    Since $t_{1},t_{2}\in\bbR$ were arbitrary, both claims follow from this.
\end{proof}

From Lemma~\ref{L2.1}, Proposition~\ref{P2.6}, and  a Gr\"{o}nwall-type argument we now obtain the following result.

\begin{corollary}\label{C2.7}
    With $C_{\alpha}$ from Lemma~\ref{L2.1}, for all $t\in\bbR$ we have
    \[
        \max\set{
            \partial_{t}^{+}L_{\mu}(\Phi_{\eps*}^{t}\Theta),
            -\partial_{t-}L_{\mu}(\Phi_{\eps*}^{t}\Theta)
        }
        \leq C_{\alpha}L_{\mu}(\Phi_{\eps*}^{t}\Theta)^{2}.
    \]
    In particular, for all  $t\in(- \frac{1}{C_{\alpha}L_{\mu}(\Theta)},
    \frac{1}{C_{\alpha}L_{\mu}(\Theta)})$ we have
    \[
        L_{\mu}(\Phi_{\eps*}^{t}\Theta)
        \leq \frac{L_{\mu}(\Theta)}{1 - C_{\alpha}L_{\mu}(\Theta)\abs{t}}.
    \]
\end{corollary}

%\begin{proof}
%    The first claim follows from Lemma~\ref{L2.1} and     Proposition~\ref{P2.6}, and the second
%    follows from the first claim and a Gr\"{o}nwall-type argument.
%\end{proof}

\begin{proposition}\label{P2.8}
Let $T_{0}\coloneqq \frac{1}{2C_{\alpha}L_{\mu}(\Theta)}$, with $C_{\alpha}$ from Lemma~\ref{L2.1}.
    There is $\Psi:=\lim_{\eps\to 0} \Psi_{\eps}\in C\left([-T_{0},T_{0}];BC(\bbR^{2};\bbR^{2})\right)$,
%    converges uniformly to some $\Psi\in C\left([-T_{0},T_{0}];C_{b}(\bbR^{2};\bbR^{2})\right)$.
    and $\Phi^{t}\coloneqq\mathrm{Id} + \Psi^{t}$
    is a measure-preserving homeomorphism for each $t\in[-T_{0},T_{0}]$ that solves \eqref{1.6}. Moreover, for each $t\in[-T_{0},T_{0}]$ we have 
    \[
        L_{\mu}(\Phi_{*}^{t}\Theta) \leq
        \sup_{\eps>0}L_{\mu}(\Phi_{\eps*}^{t}\Theta) \leq 2L_{\mu}(\Theta)
    \]
    and
    \[
        \max\set{\norm{\Phi^{t}}_{\dot{C}^{0,1}}, \norm{(\Phi^{t})^{-1}}_{\dot{C}^{0,1}}}
        \leq \sup_{\eps>0}\max\set{
            \norm{\Phi_{\eps}^{t}}_{\dot{C}^{0,1}},
            \norm{(\Phi_{\eps}^{t})^{-1}}_{\dot{C}^{0,1}}
        }
        \leq e^{2C_{\alpha}L_{\mu}(\Theta)\abs{t}}.
    \]
\end{proposition}

\begin{proof}
    Corollary~\ref{C2.7} shows that
    \[
        M\coloneqq \sup_{\eps>0}\sup_{t\in[-T_{0},T_{0}]}L_{\mu}(\Phi_{\eps*}^{t}\Theta)
        \in [L_{\mu}(\Theta), 2L_{\mu}(\Theta)] .
    \]
    We may assume that $L_{\mu}(\Theta) > 0$ because otherwise $\theta^{0} \equiv 0$
    and the result follows trivially.
    Fix any $t_{0}\in[-T_{0},T_{0}]$ and pick any $t\in[-T_{0},T_{0}]$, $\eps>0$, and
    $\eps'\in(0,\eps)$.  Then Lemmas \ref{L2.1}, \ref{L2.2}, and \ref{L2.3} show that
    \begin{equation}\label{2.5}
        \begin{aligned}
            &\norm{u_{\eps}(\theta_{\eps}^{t})\circ
            (\operatorname{Id} + \Gamma_{\eps}^{t_{0},t})
            - u_{\eps'}(\theta_{\eps'}^{t})\circ
            (\operatorname{Id} + \Gamma_{\eps'}^{t_{0},t})}_{L^{\infty}}
            \\&\quad\quad\quad
            \leq \norm{u_{\eps}(\theta_{\eps}^{t})}_{\dot{C}^{0,1}}
            \norm{\Gamma_{\eps}^{t_{0},t} - \Gamma_{\eps'}^{t_{0},t}}_{L^{\infty}}
             + \norm{u_{\eps}(\theta_{\eps}^{t})
            - u(\theta_{\eps}^{t})}_{L^{\infty}}
            \\&\quad\quad\quad\quad\quad\quad\quad\quad\quad
           + \norm{u(\theta_{\eps}^{t})
            - u(\theta_{\eps'}^{t})}_{L^{\infty}}
            + \norm{u(\theta_{\eps'}^{t})
            - u_{\eps'}(\theta_{\eps'}^{t})}_{L^{\infty}}
            \\&\quad\quad\quad
            \leq C_{\alpha}M
            \norm{\Gamma_{\eps}^{t_{0},t} - \Gamma_{\eps'}^{t_{0},t}}_{L^{\infty}}
            + C_{\alpha}M
            \norm{\Phi_{\eps}^{t} - \Phi_{\eps'}^{t}}_{L^{\infty}}
            + C_{\alpha}\norm{\theta^{0}}_{L^{\infty}} \eps^{1-2\alpha}
        \end{aligned}
    \end{equation}
    where $C_{\alpha}$ (which we now fix for the rest of the proof)
    is two times the maximum of all the $C_{\alpha}$'s appearing in those lemmas.
    Integrating \eqref{2.5} between any $t_{1},t_{2}\in[-T_{0},T_{0}]$ yields
    \begin{equation}\label{2.6}
        \begin{aligned}
            &\norm{\Gamma_{\eps}^{t_{0},t_{1}} - \Gamma_{\eps'}^{t_{0},t_{1}}
            - \Gamma_{\eps}^{t_{0},t_{2}} + \Gamma_{\eps'}^{t_{0},t_{2}}}_{L^{\infty}}
            \\&\quad\quad
            \leq C_{\alpha}M\abs{\int_{t_{2}}^{t_{1}}
            \norm{\Gamma_{\eps}^{t_{0},\tau} - \Gamma_{\eps'}^{t_{0},\tau }}_{L^{\infty}}d\tau}
            + C_{\alpha}M\abs{\int_{t_{2}}^{t_{1}}
            \norm{\Phi_{\eps}^{\tau} - \Phi_{\eps'}^{\tau}}_{L^{\infty}}d\tau}
            \\&\quad\quad\quad\quad\quad\quad\quad\quad
            + C_{\alpha}\norm{\theta^{0}}_{L^{\infty}}\abs{t_{1} - t_{2}}\eps^{1-2\alpha}.
        \end{aligned}
    \end{equation}
    In particular, taking $t_{0} = 0$, dividing by $\abs{t_{1} - t_{2}}$,
    and letting $t_{1}\to t_{2}^{\pm}$ shows that
    \begin{align*}
        \max\set{
            \partial_{t}^{+}\norm{\Psi_{\eps}^{t} - \Psi_{\eps'}^{t}}_{L^{\infty}},
            -\partial_{t-}\norm{\Psi_{\eps}^{t} - \Psi_{\eps'}^{t}}_{L^{\infty}}
        }
        \leq 2C_{\alpha}M\norm{\Psi_{\eps}^{t} - \Psi_{\eps'}^{t}}_{L^{\infty}}
        + C_{\alpha}\norm{\theta^{0}}_{L^{\infty}}\eps^{1-2\alpha}
    \end{align*}
    for each $t\in[-T_{0},T_{0}]$, and then a Gr\"{o}nwall-type argument yields
    \[
        \norm{\Phi_{\eps}^{t} - \Phi_{\eps'}^{t}}_{L^{\infty}}
        = \norm{\Psi_{\eps}^{t} - \Psi_{\eps'}^{t}}_{L^{\infty}}
        \leq \frac{\norm{\theta^{0}}_{L^{\infty}}}{2M}
        (e^{2C_{\alpha}M\abs{t}} - 1)\eps^{1-2\alpha}.
    \]
    Applying this inequality to \eqref{2.6}, 
    %with $C_{\alpha}M$'s replaced by $2C_{\alpha}M$'s,
    dividing by $\abs{t_{1} - t_{2}}$ and then sending $t_{1}\to t_{2}^{\pm}$ shows that
    \begin{align*}
       \max & \left\{ \partial_{t}^{+}\norm{\Gamma_{\eps}^{t_{0},t} - \Gamma_{\eps'}^{t_{0},t}}_{L^{\infty}},
       -\partial_{t-}\norm{\Gamma_{\eps}^{t_{0},t} - \Gamma_{\eps'}^{t_{0},t}}_{L^{\infty}} \right\}
        \\ & \qquad\qquad\qquad\qquad\qquad \leq C_{\alpha}M\norm{\Gamma_{\eps}^{t_{0},t} - \Gamma_{\eps'}^{t_{0},t}}_{L^{\infty}}
        + C_{\alpha}\norm{\theta^{0}}_{L^{\infty}}
        e^{2C_{\alpha}M|t|}\eps^{1-2\alpha}
    \end{align*}
    for all $t\in[-T_0,T_0]$,
%     $t\geq 0$, and we can obtain a similar bound for
%    $-\partial_{t-}\norm{\Gamma_{\eps}^{t_{0},t} - \Gamma_{\eps'}^{t_{0},t}}_{L^{\infty}}$
%    and $t\leq 0$, 
    so a Gr\"{o}nwall-type argument yields
    \begin{equation*}
        \norm{\Gamma_{\eps}^{t_{0},t} - \Gamma_{\eps'}^{t_{0},t}}_{L^{\infty}}
        \leq \frac{\norm{\theta^{0}}_{L^{\infty}}e^{2C_{\alpha}MT_{0}}\eps^{1-2\alpha}}{M}
        \left(e^{C_{\alpha}M\abs{t - t_{0}}} - 1\right).
    \end{equation*}
    Therefore, $\Gamma_{\eps}^{t_{0},\,\cdot\,}$ converges uniformly
    to some $\Gamma^{t_{0},\,\cdot\,}\colon [-T_{0},T_{0}]\to BC(\bbR^{2};\bbR^{2})$
    as $\eps\to 0$. 
    
    Let $\Psi^{t}\coloneqq\Gamma^{0,t}$.
    Since $(\mathrm{Id} + \Gamma_{\eps}^{t_{0},t_{1}})
    \circ(\mathrm{Id}+\Gamma_{\eps}^{t_{1},t_{0}}) = \mathrm{Id}$  for
    all $t_{0},t_{1}\in[-T_{0},T_{0}]$ and $\eps>0$, sending $\eps\to 0$ shows that
    $(\mathrm{Id} + \Gamma^{t_{0},t_{1}})
    \circ(\mathrm{Id} + \Gamma^{t_{1},t_{0}}) = \mathrm{Id}$.  In particular,
    $\Phi^{t}\coloneqq \mathrm{Id} + \Psi^{t}$ is a homeomorphism whose inverse is
    $\mathrm{Id} + \Gamma^{t,0}$.
    Also, Lemma~\ref{L2.5} and the definition of $C_{\alpha}$ show that
    $\norm{\mathrm{Id} + \Gamma_{\eps}^{t_{0},t}}_{\dot{C}^{0,1}}
    \leq e^{C_{\alpha}M\abs{t - t_{0}}}$ for all $t_{0},t\in[-T_{0},T_{0}]$ and $\eps>0$, thus
    $\max\set{\norm{\Phi^{t}}_{\dot{C}^{0,1}}, \norm{(\Phi^{t})^{-1}}_{\dot{C}^{0,1}}}
    \leq e^{C_{\alpha}M\abs{t}}$ holds for all $t\in[-T_{0},T_{0}]$.
    By Fatou's lemma we also have $L_{\mu}(\Phi_{*}^{t}\Theta)
    \leq \liminf_{\eps\to 0}L_{\mu}(\Phi_{\eps*}^{t}\Theta) \leq M$ for each $t\in[-T_{0},T_{0}]$.
   And since each $\Phi_{\eps}^{t}$ is measure-preserving,
    their uniform limit $\Phi^{t}$ is also such because for any open set $U\subseteq\bbR^{2}$ we have that
    $\mathbbm{1}_{U}\circ\Phi_{\eps}^{t}\to \mathbbm{1}_{U}\circ\Phi^{t}$ pointwise as $\eps\to 0$ .

    It remains to show that $\Phi^{t}$ satisfies \eqref{1.6}, that is,
    with $\theta^{t}\coloneqq\theta^{0}\circ (\Phi^{t})^{-1}$ we have
    \begin{equation}\label{2.7}
        \Phi^{t} = \mathrm{Id} + \int_{0}^{t}
        u(\theta^{\tau})\circ\Phi^{\tau}\,d\tau
    \end{equation}
    for each $t\in[-T_{0},T_{0}]$.
    Taking $t_{0} = 0$ and letting $\eps'\to 0^{+}$ in \eqref{2.5} yields
    \begin{align*}
        \norm{u_{\eps}(\theta_{\eps}^{t})\circ \Phi_{\eps}^{t}
        - u(\theta^{t})\circ \Phi^{t}}_{L^{\infty}}
        \leq 2C_{\alpha}M\norm{\Phi_{\eps}^{t} - \Phi^{t}}_{L^{\infty}}
        + C_{\alpha}\norm{\theta^{0}}_{L^{\infty}}\eps^{1-2\alpha},
    \end{align*}
    which shows that the right-hand side of
    \[
        \Phi_{\eps}^{t} = \mathrm{Id} + \int_{0}^{t}
        u(\theta_{\eps}^{\tau})\circ \Phi_{\eps}^{\tau}\,d\tau
    \]
    converges uniformly to the right-hand side of \eqref{2.7}
    as $\eps\to 0$. This now proves \eqref{1.6}.
\end{proof}

\begin{proposition}\label{P2.9}
    Let $(\Theta_{1},\mu_{1})$, $(\Theta_{2},\mu_{2})$ be
    generalized layer cake representations of $\theta^{0}$ and
    $\Phi_{1},\Phi_{2}\in C\left(I;C(\bbR^{2};\bbR^{2})\right)$ be solutions to \eqref{1.6}
    on a compact interval $I\ni 0$ that are  both measure-preserving homeomorphisms and
    $\sup_{t\in I}\max_{i\in\{1,2\}}L_{\mu_{i}}(\Phi_{i*}^{t}\Theta_{i}) <\infty$.
    Then $\Phi_{1} = \Phi_{2}$.
\end{proposition}

\begin{proof}
    Let
    \[
        M \coloneqq \sup_{t\in I}\max_{i=1,2}L_{\mu_{i}}(\Phi_{i*}^{t}\Theta_{i}).
    \]
    Then Lemmas \ref{L2.1} and \ref{L2.2} show that
    \begin{align*}
        \norm{u(\Phi_{1*}^{t}\theta^{0})\circ \Phi_{1}^{t}
        - u(\Phi_{2*}^{t}\theta^{0})\circ \Phi_{2}^{t}}_{L^{\infty}}
        &\leq \norm{u(\Phi_{1*}^{t}\theta^{0})}_{\dot{C}^{0,1}}
        \norm{\Phi_{1}^{t} - \Phi_{2}^{t}}_{L^{\infty}}
        + \norm{u(\Phi_{1*}^{t}\theta^{0}) - u(\Phi_{2*}^{t}\theta^{0})}_{L^{\infty}} \\
        &\leq C_{\alpha}M\norm{\Phi_{1}^{t} - \Phi_{2}^{t}}_{L^{\infty}}
    \end{align*}
    with some $C_{\alpha}$, which together with continuity of
    $\norm{\Phi_{1}^{t}-\Phi_{2}^{t}}_{L^{\infty}}$ in $t$ yields
    \[
        \max\set{
            \partial_{t}^{+}\norm{\Phi_{1}^{t} - \Phi_{2}^{t}}_{L^{\infty}},
            -\partial_{t-}\norm{\Phi_{1}^{t} - \Phi_{2}^{t}}_{L^{\infty}}
        }
        \leq C_{\alpha}M\norm{\Phi_{1}^{t} - \Phi_{2}^{t}}_{L^{\infty}}.
    \]
  A Gr\"{o}nwall-type argument  finishes the proof.
\end{proof}

Combining Propositions \ref{P2.8} and \ref{P2.9} with \eqref{1.7}, the latter showing that the time spans of maximal solutions for
any two generalized layer cake representations of $\theta^{0}$ must coincide (recall that
$\sup_{t\in J}\norm{(\Phi^{t})^{-1}}_{\dot{C}^{0,1}} < \infty$ for any compact interval $J\subseteq I$),
now yields  Theorem~\ref{T1.2}.

\section{Proof of Theorem~\ref{T1.4}(i)--(ii)}\label{S3}

Again, all constants $C_{\alpha}$ below
can change from one inequality to another, but they always only depend  on $\alpha$.
Fix any $\theta^{0}$ satisfying the hypotheses and
%Suppose that the initial data $\theta^{0}\in L^{1}(\bbR^{2})\cap L^{\infty}(\bbR^{2})$
%admits a generalized layer cake representation $(\Theta,\mu)$
%composed of simple closed curves such that $L_{\mu}(\Theta), Q(\Theta), R_{\mu}(\Theta) < \infty$.
for each $\lambda\in\mathcal{L}$, let $z^{0,\lambda}\in\operatorname{PSC}(\bbR^{2})$ be
such that $\partial\Theta^{\lambda} = \operatorname{im}(z^{0,\lambda})$.

Let $\theta\colon I\to L^{1}(\bbR^{2})\cap L^{\infty}(\bbR^{2})$ be the
Lagrangian solution to \eqref{1.1}--\eqref{1.2} from Theorem~\ref{T1.2}, with  initial data $\theta^{0}$
and  flow map $\Phi\in C\left(I;C(\bbR^{2};\bbR^{2})\right)$, and consider $T_{0}\coloneqq \frac{1}{2C_{\alpha}L_{\mu}(\Theta)}$ as in Proposition \ref{P2.8} (note that $[-T_0,T_0]\subseteq I$ because $I$ is maximal). Then since  $\Phi^{t}$ is a homeomorphism for each $t\in I$,  it follows that $\Phi_{*}^{t}\Theta$ is  composed of simple closed curves (this proves the first claim in Theorem \ref{T1.4}(i)). We denote these  $z^{t,\lambda}:=\Phi^{t}\circ z^{0,\lambda} \in\operatorname{PSC}(\bbR^{2})$, where $\Phi^{t}\circ z^{0,\lambda}\in \operatorname{CC}(\bbR^{2})$ is the curve whose representative is $\Phi^{t}\circ\tilde{z}^{0,\lambda}$ whenever  $\tilde{z}^{0,\lambda}$ is a representative of $z^{0,\lambda}$ (since $\set{z^{t,\lambda}}_{t\in I}$ is clearly a connected subset of $\operatorname{CC}(\bbR^{2})$,  \cite[Lemma~B.4]{JeoZla2} shows that each $z^{t,\lambda}$ is positively oriented).
% $\Phi^{t}\circ z^{0,\lambda}$
%(i.e., the equivalence class in $\operatorname{CC}(\bbR^{2})$
%of $\Phi^{t}\circ\tilde{z}^{0,\lambda}$ for any representative $\tilde{z}^{0,\lambda}$
%of $z^{0,\lambda}$) then .
%(Since $\set{z^{t,\lambda}}_{t\in I}$ is a connected subset of
%$\operatorname{CC}(\bbR^{2})$, \cite[Lemma~B.4]{JeoZla2} shows that
%each $z^{t,\lambda}$ is positively oriented.)

%Denote $(x_{1},x_{2})^{\perp}\coloneqq(-x_{2},x_{1})$,
Fix any $\eps>0$, and recall that $\Phi_{\eps}^{t}\coloneqq\mathrm{Id} + \Psi_{\eps}^{t}$,
where $\Psi_{\eps}^{t}$ is the solution to \eqref{2.1}.
For each $(t,\lambda)\in\bbR\times\mathcal{L}$ let $z_{\eps}^{t,\lambda} \coloneqq \Phi_{\eps}^{t}\circ z^{0,\lambda}$
and $\theta_{\eps}^{t}\coloneqq\theta^{0}\circ(\Phi_{\eps}^{t})^{-1}$, then fix any arclength parametrization of $z_{\eps}^{t,\lambda}$ (we denote it again $z_{\eps}^{t,\lambda}(\cdot)$) and for $s\in[0,\ell(z_{\eps}^{t,\lambda})]$ define\begin{itemize}
    \item $\ell_{\eps}^{t,\lambda}\coloneqq \ell(z_{\eps}^{t,\lambda})$,

    \item $\mathbf{T}_{\eps}^{t,\lambda}(s) \coloneqq \partial_{s}z_{\eps}^{t,\lambda}(s)$,

    \item $\mathbf{N}_{\eps}^{t,\lambda}(s) \coloneqq \mathbf{T}^{t,\lambda}(s)^{\perp}$,

    \item $\kappa_{\eps}^{t,\lambda}(s) \coloneqq
    \partial_{s}^{2}z_{\eps}^{t,\lambda}(s)\cdot \mathbf{N}^{t,\lambda}(s)$,
    
    \item $\Delta_{\eps}^{t,\lambda,\lambda'} \coloneqq
    \Delta(z_{\eps}^{t,\lambda}, z_{\eps}^{t,\lambda'})$,

    \item $u_{\eps}^{t,\lambda}(s) \coloneqq
    u_{\eps}(\theta_{\eps}^{t}; z_{\eps}^{t,\lambda}(s))$.
\end{itemize}
Proposition \ref{P2.8} shows that $\lim_{\eps\to 0} z_\eps^{t,\lambda}= z^{t,\lambda}$ in $\operatorname{CC}(\bbR^{2})$, and
as noted in \cite[Section~4]{JeoZla2}, 
\[
    \partial_{s}^{2}z_{\eps}^{t,\lambda}(s) = \partial_{s}\mathbf{T}_{\eps}^{t,\lambda}(s)
    = \kappa_{\eps}^{t,\lambda}(s)\mathbf{N}_{\eps}^{t,\lambda}(s)
    \qquad\textrm{and}\qquad
    \partial_{s}\mathbf{N}_{\eps}^{t,\lambda}(s)
    = -\kappa_{\eps}^{t,\lambda}(s)\mathbf{T}_{\eps}^{t,\lambda}(s)
\]
holds as well.
Then the argument in \cite[Lemma~4.1]{JeoZla2} also applies here, and we obtain
\begin{equation}\label{3.2}
    \begin{aligned}
        \partial_{t}\norm{z_{\eps}^{t,\lambda}}_{\dot{H}^{2}}^{2}
        &= -3\int_{\ell_{\eps}^{t,\lambda}\bbT}\kappa_{\eps}^{t,\lambda}(s)^{2}
        \left(\partial_{s}u_{\eps}^{t,\lambda}(s) \cdot \mathbf{T}_{\eps}^{t,\lambda}(s)\right)ds
        \\&\qquad\qquad        
        + 2\int_{\ell_{\eps}^{t,\lambda}\bbT}\kappa_{\eps}^{t,\lambda}(s)
        \left(\partial_{s}^{2}u_{\eps}^{t,\lambda}(s) \cdot \mathbf{N}_{\eps}^{t,\lambda}(s)\right)ds.
    \end{aligned}
\end{equation}

In the proof of this we fix
%any $(t,\lambda)\in\bbR\times\mathcal{L}$ and 
some {\it constant-speed parametrization}
$\tilde{z}_{\eps}^{t,\lambda}\colon\bbT\to\bbR^{2}$ of $z_{\eps}^{t,\lambda}$ (that is,  $\tilde\gamma(\ell(z_{\eps}^{t,\lambda})\,\cdot)$ for some arclength parametrization $\tilde\gamma$ of $z_{\eps}^{t,\lambda}$), and for each $h\in\bbR$
let $\tilde{z}_{\eps}^{t+h,\lambda}\coloneqq
\Phi_{\eps}^{t+h}\circ(\Phi_{\eps}^{t})^{-1}\circ\tilde{z}_{\eps}^{t,\lambda}$. Since \eqref{2.2} is globally well-posed
in $BC^{2}(\bbR^{2};\bbR^{2})$ (see the paragraph before Lemma~\ref{L2.5}),
it easily follows that $\tilde{z}_{\eps}^{t,\lambda}\in H^{2}(\bbT;\bbR^{2})$ and
\begin{equation}\label{3.1}
    \tilde{z}_{\eps}^{t+h,\lambda}
    = \tilde{z}_{\eps}^{t,\lambda} + \int_{t}^{t+h}u_{\eps}(\theta_{\eps}^{\tau})
    \circ \tilde{z}_{\eps}^{\tau,\lambda}\,d\tau
\end{equation}
holds for all $h\in\bbR$ (in $H^{2}(\bbT;\bbR^{2})$).
%This is also used to evaluate the time derivative of $\ell_{\eps}^{t,\lambda}$.

\medskip
\textit{Remark}. Below we will also consider the above setup with initial data being $(\theta^{t_0},z^{t_0})$ for some $t_0\in I$ instead of $(\theta^0,z^0)$.  Since we do not yet know whether the $z^{t_{0},\lambda}$ are $H^{2}$ curves when $t_0\neq 0$,  we cannot define $\kappa_{\eps}^{t,\lambda}$ for these initial data and also cannot yet claim \eqref{3.1} to hold in $H^{2}(\bbT;\bbR^{2})$.  However, since $\Phi^{t_{0}}$ is Lipschitz by Proposition \ref{P2.8}, so are the $z^{t_{0},\lambda}$ and then
\eqref{3.1}  holds in $C^{0,1}(\bbT;\bbR^{2})$.  We will use this in the following results, up to Lemma~\ref{L3.4}.

\begin{lemma}\label{L3.1}
 With $C_{\alpha}$ from Lemma~\ref{L2.1}, for any $(t,\lambda)\in [-T_0,T_0]\times\mathcal{L}$ and $\eps>0$ we have
     \begin{equation}\label{3.3}
        \abs{\partial_{t}\ell_{\eps}^{t,\lambda}}
%        \leq \norm{u_{\eps}(\theta_{\eps}^{t})}_{\dot{C}^{0,1}}       \ell_{\eps}^{t,\lambda}
        \leq 2C_{\alpha}L_{\mu}(\Theta)\ell_{\eps}^{t,\lambda}
    \end{equation}
\end{lemma}

\begin{proof}
%%and the integral is taken with respect to the norm of $H^{2}(\bbT;\bbR^{2})$.
%    Let $(t,\lambda)\in\bbR\times\mathcal{L}$, fix a constant-speed parametrization
%    $\tilde{z}^{t,\lambda}\colon\bbT\to\bbR^{2}$ of $z^{t,\lambda}$
%    and define $\tilde{z}^{\tau,\lambda}$ for each $\tau\in\bbR$ as in \eqref{3.1}.
   With $\tilde{z}_{\eps}^{t,\lambda}$ as above, we have $\abs{\partial_{\xi}\tilde{z}_{\eps}^{t,\lambda}(\xi)}
    = \ell_{\eps}^{t,\lambda}> 0$ for any $(t,\lambda)\in\bbR\times\mathcal{L}$.  Then for any $(h,\xi)\in\bbR\times\bbT$ we get
    %, \eqref{3.1} shows
    \begin{equation*}
        \begin{aligned}
            &\abs{\partial_{\xi}\tilde{z}_{\eps}^{t+h,\lambda}(\xi)}
            - \abs{\partial_{\xi}\tilde{z}_{\eps}^{t,\lambda}(\xi)}
            \\&\quad\quad\quad
            = \frac{2\int_{t}^{t+h} \frac d{d\xi} u_{\eps}(\theta_{\eps}^{\tau};
            \tilde{z}_{\eps}^{\tau,\lambda}(\xi))
            \cdot \partial_{\xi}\tilde{z}_{\eps}^{t,\lambda}(\xi)\,d\tau}
            {\abs{\partial_{\xi}\tilde{z}_{\eps}^{t+h,\lambda}(\xi)}
            + \abs{\partial_{\xi}\tilde{z}_{\eps}^{t,\lambda}(\xi)}}
            + \frac{\abs{\int_{t}^{t+h} \frac d{d\xi} u_{\eps}(\theta_{\eps}^{\tau};
            \tilde{z}_{\eps}^{\tau,\lambda}(\xi))\,d\tau}^{2}}
            {\abs{\partial_{\xi}\tilde{z}_{\eps}^{t+h,\lambda}(\xi)}
            + \abs{\partial_{\xi}\tilde{z}_{\eps}^{t,\lambda}(\xi)}}.
        \end{aligned}
    \end{equation*}
    Since $\norm{D(u_{\eps}(\theta_{\eps}^{\tau}))}_{L^{\infty}}$ is continuous in $\tau$
    by Lemma~\ref{L2.5}, integrating in $\xi$, dividing by $h$,
    and then letting $h\to 0$ shows that
    \[
        \partial_{t}\ell_{\eps}^{t,\lambda}
        = \int_{\bbT}
        \frac d{d\xi} u_{\eps}(\theta_{\eps}^{t};\tilde{z}_{\eps}^{t,\lambda}(\xi))
        \cdot \frac{\partial_{\xi}\tilde{z}_{\eps}^{t,\lambda}(\xi)}
        {\abs{\partial_{\xi}\tilde{z}_{\eps}^{t,\lambda}(\xi)}}\,d\xi         = \int_{\ell_{\eps}^{t,\lambda}\bbT}
        \partial_{s}u_{\eps}^{t,\lambda}(s)\cdot\mathbf{T}_{\eps}^{t,\lambda}(s)\,ds.
    \]
%        For any $(t,\lambda)\in\bbR\times\mathcal{L}$ we have
%    \[
%        \partial_{t}\ell_{\eps}^{t,\lambda}
%        = \int_{\ell_{\eps}^{t,\lambda}\bbT}
%        \partial_{s}u_{\eps}^{t,\lambda}(s)\cdot\mathbf{T}_{\eps}^{t,\lambda}(s)\,ds.
%    \]
%where we used a change of variables at the end.  
The result now follows from Lemma \ref{L2.1} and Corollary~\ref{C2.7}.
\end{proof}

\begin{lemma}\label{L3.2}
    With $C_{\alpha}$ from Lemma~\ref{L2.1}, for any $(t,\lambda)\in I\times\mathcal{L}$ and all small enough $h\in\bbR$ we have
    \[
        e^{-3C_{\alpha}L_{\mu}(\Phi_{*}^{t}\Theta)\abs{h}}\ell(z^{t,\lambda})
        \leq \ell(z^{t+h,\lambda}) \leq
        e^{3C_{\alpha}L_{\mu}(\Phi_{*}^{t}\Theta)\abs{h}}\ell(z^{t,\lambda}).
    \]
%    holds for any small enough $h\in\bbR$.
\end{lemma}

\begin{proof}
   Since $\ell\colon\operatorname{CC}(\bbR^{2})\to[0,\infty]$ is lower semi-continuous
    (by definition) and $\lim_{\eps\to 0} z_\eps^{t,\lambda}= z^{t,\lambda}$ in $\operatorname{CC}(\bbR^{2})$, a Gr\"{o}nwall-type argument applied to \eqref{3.3} shows that for $\abs{t}\leq T_0$ we have
    \begin{equation}\label{3.4}
        \ell(z^{t,\lambda}) \leq e^{2C_{\alpha}L_{\mu}(\Theta)\abs{t}}\ell(z^{0,\lambda}).
    \end{equation}

    Next fix any $t_{0}\in I$ with $\abs{t_{0}}\leq \frac {T_0}2$.
    Repeating the proof of \eqref{3.4} with $z^{t_0,\lambda}$ in place of $z^{0,\lambda}$ (recall  the remark after \eqref{3.1}) shows that
    %Considering $\theta^{t_{0}}$ as the initial condition for \eqref{1.1}--\eqref{1.2}
    %at time $ t_{0}$ and then repeating the proof of \eqref{3.4} shows that 
    \[
        \ell(z^{t,\lambda}) \leq e^{2C_{\alpha}L_\mu(\Phi_{*}^{t_{0}}\Theta)\abs{t - t_{0}}}\ell(z^{t_0,\lambda})
    \]
    %\eqref{3.4} holds with $(z^{t_{0},\lambda},L_\mu(\Phi_{*}^{t_{0}}\Theta)\abs{t - t_{0}})$
    %in place of $(z^{0,\lambda},L_\mu(\Theta)\abs{t})$
    whenever $\abs{t - t_{0}}\leq\frac{1}{2C_{\alpha}L_{\mu}(\Phi_{*}^{t_{0}}\Theta)}$.
    %We do not yet know whether \eqref{3.1} holds in $H^{2}(\bbT;\bbR^{2})$ without all the $\eps$ because $\Phi^{t}$ is only Lipschitz,  but  all the involved curves are rectifiable and
    %\eqref{3.1} without $\eps$ does hold in $C^{0,1}(\bbT;\bbR^{2})$.
    %, and that is enough  for further arguments.
    Since $L_{\mu}(\Phi_{*}^{t_{0}}\Theta) \leq \frac{4}{3}L_{\mu}(\Theta)$ by Corollary~\ref{C2.7},
    we see that
    %can substitute $0$ for $t$ which shows
    \[
        \ell(z^{0,\lambda}) \leq e^{3C_{\alpha}L_{\mu}(\Theta)\abs{t_{0}}}\ell(z^{t_{0},\lambda}),
    \]
and so
    \[
        e^{-3C_{\alpha}L_{\mu}(\Theta)\abs{t}}\ell(z^{0,\lambda})
        \leq \ell(z^{t,\lambda}) \leq
        e^{3C_{\alpha}L_{\mu}(\Theta)\abs{t}}\ell(z^{0,\lambda})
    \]
    holds whenever $\abs{t}\leq \frac {T_0}2$.
    Applying this with $(\theta^t,t,h)$ in place of $(\theta^0,0,t)$,
    for any $t\in I$, now shows the claim.
\end{proof}

\begin{proposition}\label{P3.3}
    $R_{\mu}(\Phi_{*}^{t}\Theta)$ is continuous in $t$.   In particular,
    for any compact interval $J\subseteq I$ we have
    \[
        \sup_{t\in J}R_{\mu}(\Phi_{*}^{t}\Theta) < \infty.
    \]
\end{proposition}

\begin{proof}
    Take any compact interval $J\subseteq I$ containing $0$.
%    With $T_{0}$ as in Proposition~\ref{P2.8},
    Lemmas~\ref{L2.1}, \ref{L2.5} and Corollary~\ref{C2.7} show that
    \begin{equation}\label{3.5}
        \abs{\Phi_{\eps}^{t_{1}}(x) - \Phi_{\eps}^{t_{1}}(y)
        - \Phi_{\eps}^{t_{2}}(x) + \Phi_{\eps}^{t_{2}}(y)}
        \leq (e^{2C_{\alpha}L_{\mu}(\Theta)\abs{t_{1}-t_{2}}} - 1)
        \abs{\Phi_{\eps}^{t_{1}}(x) - \Phi_{\eps}^{t_{1}}(y)}
    \end{equation}
    for all $\eps>0$, $t_{1},t_{2}\in[-T_{0},T_{0}]$ and $x,y\in\bbR^{2}$.
    Hence, taking $\eps\to 0$ shows that \eqref{3.5} continues to hold
    with $\Phi$ in place of $\Phi_{\eps}$. 
    Applying the same argument with $\Phi_{*}^{t}\Theta$ in place of $\Theta$ (recall the remark after \eqref{3.1}), shows that
    each $t\in J$ has a neighborhood such that
    \begin{equation}\label{3.6}
        \abs{\Phi^{t_{1}}(x) - \Phi^{t_{1}}(y)
        - \Phi^{t_{2}}(x) + \Phi^{t_{2}}(y)}
        \leq (e^{2C_{\alpha}M\abs{t_{1}-t_{2}}} - 1)
        \abs{\Phi^{t_{1}}(x) - \Phi^{t_{1}}(y)}
    \end{equation}
    holds for any $t_{1},t_{2}\in J$ in that neighborhood, where
    \[
        M\coloneqq \sup_{t\in J}L_{\mu}(\Phi_{*}^t\Theta) < \infty.
    \]
 Since $J$ is an interval, it follows that \eqref{3.6} in fact holds for
    all $t_{1},t_{2}\in J$.

    Fix $t_{1},t_{2}\in J$, $\lambda,\lambda'\in\mathcal{L}$, and $\eta>0$.
    Then there are $x\in\operatorname{im}(z^{0,\lambda})$ and
    $y\in\operatorname{im}(z^{0,\lambda'})$ such that
    $\Delta(z^{t_{1},\lambda},z^{t_{1},\lambda'}) = \abs{\Phi^{t_{1}}(x) - \Phi^{t_{1}}(y)}$.
    A similar argument as in \eqref{2.4} now shows that
    \begin{align*}
        &\frac{1}{\left(
            \Delta(z^{t_{1},\lambda},z^{t_{1},\lambda'}) + \eta
        \right)^{2\alpha}}
        - \frac{1}{\left(
            \Delta(z^{t_{2},\lambda},z^{t_{2},\lambda'}) + \eta
        \right)^{2\alpha}}
        \\&\quad\quad\quad\quad
        \leq \frac{\abs{\Phi^{t_{1}}(x) - \Phi^{t_{1}}(y)
        - \Phi^{t_{2}}(x) + \Phi^{t_{2}}(y)}}
        {\left(\abs{\Phi^{t_{1}}(x) - \Phi^{t_{1}}(y)} + \eta\right)
        \left(\abs{\Phi^{t_{2}}(x) - \Phi^{t_{2}}(y)} + \eta\right)^{2\alpha}}
        \\&\quad\quad\quad\quad
        \leq \frac{e^{2C_{\alpha}M\abs{t_{1} - t_{2}}} - 1}
        {(\Delta(z^{t_{2},\lambda},z^{t_{2},\lambda'}) + \eta)^{2\alpha}},
    \end{align*}
    thus
    \begin{equation}\label{3.7}
        \frac{1}{\left(
            \Delta(z^{t_{1},\lambda},z^{t_{1},\lambda'}) + \eta
        \right)^{2\alpha}}
        \leq \frac{e^{2C_{\alpha}M\abs{t_{1} - t_{2}}}}
        {\Delta(z^{t_{2},\lambda},z^{t_{2},\lambda'})^{2\alpha}}.
    \end{equation}
    Lemma~\ref{L3.2} yields
    \[
        e^{-3C_{\alpha}M\abs{t_{1} - t_{2}}}\ell(z^{t_{2},\lambda})
        \leq \ell(z^{t_{1},\lambda}) \leq
        e^{3C_{\alpha}M\abs{t_{1} - t_{2}}}\ell(z^{t_{2},\lambda})
    \]
    for any $t_{1},t_{2}\in J$ because $J$ is an interval, and so we get
    \begin{equation}\label{3.8}
        \frac{\ell(z^{t_{1},\lambda})^{1/2}}
        {\ell(z^{t_{1},\lambda'})^{1/2}\left(
            \Delta(z^{t_{1},\lambda},z^{t_{1},\lambda'}) + \eta
        \right)^{2\alpha}}
        \leq \frac{e^{5C_{\alpha}M\abs{t_{1} - t_{2}}}\ell(z^{t_{2},\lambda})^{1/2}}
        {\ell(z^{t_{2},\lambda'})^{1/2}
        \Delta(z^{t_{2},\lambda},z^{t_{2},\lambda'})^{2\alpha}}.
    \end{equation}
    Letting $\eta\to 0^{+}$, integrating in $\lambda'$,
    and then taking the supremum over $\lambda\in\calL$ shows that
    \begin{equation}\label{3.9}
        R_{\mu}(\Phi_{*}^{t_{1}}\Theta)
        \leq e^{5C_{\alpha}M\abs{t_{1} - t_{2}}}R_{\mu}(\Phi_{*}^{t_{2}}\Theta).
    \end{equation}
    Since $t_{1},t_{2}\in J$ were arbitrary, the claim follows.
\end{proof}

This proves the second claim in Theorem \ref{T1.4}(i). Our proof of Theorem \ref{T1.4}(ii) will use the following
extension to $R_{\mu}(\Phi_{\eps*}^{t}\Theta)$ of the bound on $R_{\mu}(\Phi_{*}^{t}\Theta)$ from the last proof.
% of Proposition~\ref{P3.3} to $\eps>0$.

\begin{lemma}\label{L3.4}
    $R_{\mu}(\Phi_{\eps*}^{t}\Theta)$ is continuous in $t$ for each $\eps>0$, and
    \begin{equation}\label{3.10}
        R_{\mu}(\Phi_{\eps*}^{t}\Theta) \leq e^{4C_{\alpha}L_{\mu}(\Theta)\abs{t}}R_{\mu}(\Theta)
        \leq 8R_{\mu}(\Theta)
    \end{equation}
holds for any $t\in[-T_{0},T_{0}]$,  with    $C_{\alpha}$ from Lemma~\ref{L2.1}.
\end{lemma}

\begin{proof}
    We have \eqref{3.3} for $t\in[-T_{0},T_{0}]$, so a Gr\"{o}nwall-type argument shows that
    \[
        e^{-2C_{\alpha}L_{\mu}(\Theta)\abs{t_{1} - t_{2}}}\ell_{\eps}^{t_{2},\lambda}
        \leq \ell_{\eps}^{t_{1},\lambda} \leq
        e^{2C_{\alpha}L_{\mu}(\Theta)\abs{t_{1} - t_{2}}}\ell_{\eps}^{t_{2},\lambda}
    \]
    for any $t_{1},t_{2}\in[-T_{0},T_{0}]$ and $\lambda\in\mathcal{L}$.
    Since \eqref{3.5} holds, a similar argument as in \eqref{3.7} shows
    \[
        \frac{1}{\left(
            \Delta_{\eps}^{t_{1},\lambda,\lambda'} + \eta
        \right)^{2\alpha}}
        \leq \frac{e^{2C_{\alpha}L_{\mu}(\Theta)\abs{t_{1} - t_{2}}}}
        {(\Delta_{\eps}^{t_{2},\lambda,\lambda'})^{2\alpha}}
    \]
    for any $t_{1},t_{2}\in[-T_{0},T_{0}]$, $\lambda,\lambda'\in\mathcal{L}$ and $\eta>0$.
    Hence we see that \eqref{3.8}, and thus also \eqref{3.9}, continue to hold with
    $(z_{\eps},4C_{\alpha}L_{\mu}(\Theta),\Phi_{\eps})$ in place of $(z,5C_{\alpha}M,\Phi)$.
    This now shows both claims 
    (the second inequality in \eqref{3.10} follows by the definition of $T_{0}$
    and $e^{2}\leq 8$).
\end{proof}

We now turn to  Theorem \ref{T1.4}(ii),
%showing that $Q(\Phi_{*}^{t}\Theta) < \infty$ for all small enough $t$.
which we prove by using Lemma \ref{L3.1} and 
%Since $\ell(\cdot)$ and $\norm{\cdot}_{\dot{H}^{2}}$ are both
%lower semi-continuous on $\operatorname{CC}(\bbR^{2})$
%(the latter by \cite[Corollary~B.3]{JeoZla2}),
%so is $\ell(\cdot)\norm{\cdot}_{\dot{H}^{2}}^{2}$.  Therefore
%it is enough to establish a uniform-in-$(\eps,\lambda)$ bound on
%$\ell(z_{\eps}^{t,\lambda})\norm{z_{\eps}^{t,\lambda}}_{\dot{H}^{2}}^{2}$ for all $t$ close to 0.
%We do this 
by estimating the right-hand side of \eqref{3.2}
in lemmas  below.
%, via the following lemmas. 

\begin{lemma}\label{L3.5}
    There is $C_{\alpha}$ such that for any $\beta\in(0,1]$, any $C^{1,\beta}$ closed curve
    $\gamma\colon\ell\bbT\to\bbR^{2}$ parametrized by arclength, and any $x\in\bbR^{2}$ we have
    \[
        \int_{\ell\bbT}\frac{ds}{\abs{x - \gamma(s)}^{1+2\alpha}}
        \leq C_{\alpha}\frac{\ell\norm{\gamma}_{\dot{C}^{1,\beta}}^{1/\beta}}
        {d(x,\operatorname{im}(\gamma))^{2\alpha}}.
    \]
\end{lemma}

\begin{proof}
    Let $d\coloneqq\frac{1}{4}\norm{\gamma}_{\dot{C}^{1,\beta}}^{-1/\beta}$
    and $\Delta\coloneqq d(x,\operatorname{im}(\gamma))$.
    Then \cite[Lemma~A.2]{JeoZla2} shows that
    \begin{align*}
        \int_{\ell\bbT}\frac{ds}{\abs{x - \gamma(s)}^{1+2\alpha}}
        &\leq \frac{\ell}{4d}\left(
            \int_{\abs{s}\leq\Delta}\frac{ds}{\Delta^{1+2\alpha}}
            + \int_{\Delta\leq\abs{s}\leq 2d}\frac{ds}{\abs{s/2}^{1+2\alpha}}
        \right)
        + \frac{1}{\Delta^{2\alpha}}\int_{\ell\bbT}\frac{ds}{d} \\
        &\leq \frac{\ell}{2d\Delta^{2\alpha}}
        + \frac{\ell}{2^{1-2\alpha}\alpha d\Delta^{2\alpha}}
        + \frac{\ell}{d\Delta^{2\alpha}}
        = C_{\alpha}\frac{\ell\norm{\gamma}_{\dot{C}^{1,\beta}}^{1/\beta}}
        {\Delta^{2\alpha}}.
    \end{align*}
\end{proof}

\begin{lemma}\label{L3.6}
    For any $\theta\in L^{1}(\bbR^{2})\cap L^{\infty}(\bbR^{2})$ with
    a generalized layer cake representation $(\Theta,\mu)$
    such that $L_{\mu}(\Theta)<\infty$, and for any $\eps>0$ and $x,h_{1},h_{2}\in\bbR^{2}$, we have
    \begin{align*}
        D^{2}(u_{\eps}(\theta))(x)(h_{1},h_{2})
        &= \int_{\mathcal{L}}\int_{\bbR^{2}}
        D^{2}(\nabla^{\perp}K_{\eps})(x - y)(h_{1},h_{2})
        \left(\mathbbm{1}_{\Theta^{\lambda}}(y) - \mathbbm{1}_{\Theta^{\lambda}}(x)\right)
        dy\,d\mu(\lambda) \\
        &= \int_{\mathcal{L}}\int_{\Theta^{\lambda}}
        D^{2}(\nabla^{\perp}K_{\eps})(x - y)(h_{1},h_{2})
        \,dy\,d\mu(\lambda).
    \end{align*}
    Moreover, there is $C_{\alpha}$  such that
    \[
        \int_{\mathcal{L}}\int_{\bbR^{2}}
        \abs{D^{2}(\nabla^{\perp}K_{\eps})(x - y)}
        \abs{\mathbbm{1}_{\Theta^{\lambda}}(y) - \mathbbm{1}_{\Theta^{\lambda}}(x)}
        d|\mu|(\lambda)\,dy
        \leq \frac{C_{\alpha}}{\eps}
        \int_{\mathcal{L}}\frac{d|\mu|(\lambda)}{d(x,\partial\Theta^{\lambda})^{2\alpha}}.
    \]
\end{lemma}

\begin{proof}
    Oddness of $D^{2}(\nabla^{\perp}K_{\eps})$ shows that
    \begin{align*}
        D^{2}(u_{\eps}(\theta))(x)(h_{1},h_{2})
        &= \int_{\bbR^{2}}D^{2}(\nabla^{\perp}K_{\eps})(x - y)(h_{1},h_{2})
        (\theta(y) - \theta(x))\,dy \\
        &= \int_{\bbR^{2}}\int_{\mathcal{L}}
        D^{2}(\nabla^{\perp}K_{\eps})(x - y)(h_{1},h_{2})
        \left(\mathbbm{1}_{\Theta^{\lambda}}(y) - \mathbbm{1}_{\Theta^{\lambda}}(x)\right)
        d\mu(\lambda)\,dy.
    \end{align*}
    Then proceeding as in Lemma~\ref{L2.1} and using 
    $\abs{D^{2}(\nabla^{\perp}K_{\eps})(x - y)}
    \leq \frac{C_{\alpha}}{\eps\abs{x - y}^{2+2\alpha}}$
    in place of $\abs{D(\nabla^{\perp}K_{\eps})(x - y)}
    \leq \frac{C_{\alpha}}{\abs{x - y}^{2+2\alpha}}$ proves the second claim.
    Fubini's theorem now yields the first equality of the first claim,
    and the second one follows by oddness of $D^{2}(\nabla^{\perp}K_{\eps})$.
\end{proof}

\begin{lemma}\label{L3.7}
    There is $C_{\alpha}$ such that for each $(t,\lambda)\in\bbR\times\mathcal{L}$ and $\eps>0$ we have
    \begin{equation}\label{3.11}
        \abs{\partial_{t}\norm{z_{\eps}^{t,\lambda}}_{\dot{H}^{2}}^{2}}
        \leq C_{\alpha}(L_{\mu}(\Phi_{\eps*}^{t}\Theta) + R_{\mu}(\Phi_{\eps*}^{t}\Theta)) \,
        Q(\Phi_{\eps*}^{t}\Theta)\norm{z_{\eps}^{t,\lambda}}_{\dot{H}^{2}}^{2}
    \end{equation}
\end{lemma}

\begin{proof}
%    All constants $C_{\alpha}$ in this proof depend only on $\alpha$ and
    %can change from one inequality to another. 
    With $z_{\eps}^{t,\lambda}(\cdot)$ being
    the previously fixed arclength parametrization of $z_{\eps}^{t,\lambda}$, we have
    \begin{align*}
        \partial_{s}^{2}u_{\eps}^{t,\lambda}(s)
        &= D^{2}(u_{\eps}(\theta_{\eps}^{t}))(z_{\eps}^{t,\lambda}(s))
        (\mathbf{T}_{\eps}^{t,\lambda}(s), \mathbf{T}_{\eps}^{t,\lambda}(s))
        + \kappa_{\eps}^{t,\lambda}(s)D(u_{\eps}(\theta_{\eps}^{t}))
        (z_{\eps}^{t,\lambda}(s))(\mathbf{N}_{\eps}^{t,\lambda}(s)).
    \end{align*}
    Hence \eqref{3.2} yields
    \begin{equation}\label{3.12}
        \begin{aligned}
            \partial_{t}\norm{z_{\eps}^{t,\lambda}}_{\dot{H}^{2}}^{2} &=
            \int_{\ell_{\eps}^{t,\lambda}\bbT}\kappa_{\eps}^{t,\lambda}(s)^{2}
            \left[
                2D(u_{\eps}(\theta_{\eps}^{t}))(z_{\eps}^{t,\lambda}(s))
                (\mathbf{N}_{\eps}^{t,\lambda}(s))
                \cdot \mathbf{N}_{\eps}^{t,\lambda}(s)
                - 3\,\partial_{s}u_{\eps}^{t,\lambda}(s)\cdot\mathbf{T}_{\eps}^{t,\lambda}(s)
            \right] ds
            \\&\quad\quad
            + 2\int_{\ell_{\eps}^{t,\lambda}\bbT}\kappa_{\eps}^{t,\lambda}(s)
            \left[
                D^{2}(u_{\eps}(\theta_{\eps}^{t}))(z_{\eps}^{t,\lambda}(s))
                (\mathbf{T}_{\eps}^{t,\lambda}(s),\mathbf{T}_{\eps}^{t,\lambda}(s))
                \cdot \mathbf{N}_{\eps}^{t,\lambda}(s)
            \right] ds.
        \end{aligned}
    \end{equation}
    Lemma~\ref{L2.1} shows that the absolute value of the first integral is bounded by
    \[
        5\norm{u_{\eps}(\theta_{\eps}^{t})}_{\dot{C}^{0,1}}
        \norm{z_{\eps}^{t,\lambda}}_{\dot{H}^{2}}^{2}
        \leq C_{\alpha}L_{\mu}(\Phi_{\eps*}^{t}\Theta)\norm{z_{\eps}^{t,\lambda}}_{\dot{H}^{2}}^{2}
        \leq C_{\alpha}L_{\mu}(\Phi_{\eps*}^{t}\Theta)
        Q(\Phi_{\eps*}^{t}\Theta)\norm{z_{\eps}^{t,\lambda}}_{\dot{H}^{2}}^{2},
    \]
    where the second inequality holds by
    $Q(\Phi_{\eps*}^{t}\Theta)\geq 4$, which in turn follows from \cite[Lemma~A.1]{JeoZla2}.
    Hence, it remains to estimate the second integral, which we denote $G_{1}$.
    We will suppress $t$ from the notation for the sake of simplicity
    because  it will be fixed in the arguments below.
    
    Since $L_{\mu}(\Phi_{\eps*}\Theta)<\infty$, Lemma~\ref{L3.6},
    Fubini's theorem, and Green's theorem show that
    \begin{align*}
        G_{1} &=
        \int_{\ell_{\eps}^{\lambda}\bbT}\int_{\mathcal{L}}\int_{\Phi_{\eps}(\Theta^{\lambda'})}
        \kappa_{\eps}^{\lambda}(s)\mathbf{N}_{\eps}^{\lambda}(s) \cdot
        D^{2}(\nabla^{\perp}K_{\eps})(z_{\eps}^{\lambda}(s) - y)
        (\mathbf{T}_{\eps}^{\lambda}(s), \mathbf{T}_{\eps}^{\lambda}(s))
 %       \mathbbm{1}_{\Phi_{\eps}(\Theta^{\lambda'})}(y)
        \,dy\,d\mu(\lambda')\,ds\\
        &= \int_{\mathcal{L}}\int_{\ell_{\eps}^{\lambda}\bbT}\int_{\Phi_{\eps}(\Theta^{\lambda'})}
        \kappa_{\eps}^{\lambda}(s)\mathbf{N}_{\eps}^{\lambda}(s) \cdot
        D^{2}(\nabla^{\perp}K_{\eps})(z_{\eps}^{\lambda}(s) - y)
        (\mathbf{T}_{\eps}^{\lambda}(s), \mathbf{T}_{\eps}^{\lambda}(s))
 %       \mathbbm{1}_{\Phi_{\eps}(\Theta^{\lambda'})}(y)
        \,dy\,ds\,d\mu(\lambda')\\
        &= -\int_{\mathcal{L}}\int_{\ell_{\eps}^{\lambda}\bbT\times\ell_{\eps}^{\lambda'}\bbT}
        \kappa_{\eps}^{\lambda}(s)D^{2}K_{\eps}(z_{\eps}^{\lambda}(s) - z_{\eps}^{\lambda'}(s'))
        (\mathbf{T}_{\eps}^{\lambda}(s),\mathbf{T}_{\eps}^{\lambda}(s))
        (\mathbf{T}_{\eps}^{\lambda'}(s')\cdot \mathbf{N}_{\eps}^{\lambda}(s))
        \,ds\,ds'd\mu(\lambda').
    \end{align*}
    Note that the last integrand is jointly measurable in $(s,s')$ but not necessarily in $(s,s',\lambda')$ because the parametrizations
    $z_{\eps}^{\lambda'}(\cdot)$ are chosen independently from each other.
    From
    \[
        \mathbf{T}_{\eps}^{\lambda}(s) =
        (\mathbf{T}_{\eps}^{\lambda'}(s')\cdot\mathbf{T}_{\eps}^{\lambda}(s))\mathbf{T}_{\eps}^{\lambda'}(s')
        + (\mathbf{N}_{\eps}^{\lambda'}(s')\cdot\mathbf{T}_{\eps}^{\lambda}(s))\mathbf{N}_{\eps}^{\lambda'}(s')
    \]
    and $\mathbf{N}_{\eps}^{\lambda'}(s')\cdot\mathbf{T}_{\eps}^{\lambda}(s)
    = -\mathbf{T}_{\eps}^{\lambda'}(s')\cdot\mathbf{N}_{\eps}^{\lambda}(s)$
    it now follows that $\abs{G_{1}}$ is bounded by the sum of 
    \begin{align*}
        G_{2} &\coloneqq \overline{\int_{\mathcal{L}}}
        \bigg|\int_{\ell_{\eps}^{\lambda}\bbT\times\ell_{\eps}^{\lambda'}\bbT}
        \kappa_{\eps}^{\lambda}(s)
        D^{2}K_{\eps}(z_{\eps}^{\lambda}(s) - z_{\eps}^{\lambda'}(s'))
        (\mathbf{T}_{\eps}^{\lambda}(s), \mathbf{T}_{\eps}^{\lambda'}(s'))
        \\&\qquad\qquad\qquad\qquad
        (\mathbf{T}_{\eps}^{\lambda'}(s') \cdot \mathbf{N}_{\eps}^{\lambda}(s))
        (\mathbf{T}_{\eps}^{\lambda'}(s') \cdot \mathbf{T}_{\eps}^{\lambda}(s))
        \,ds\,ds'\bigg|\,d|\mu|(\lambda'), \\
        G_{3} &\coloneqq \overline{\int_{\mathcal{L}}}
        \bigg|\int_{\ell_{\eps}^{\lambda}\bbT\times\ell_{\eps}^{\lambda'}\bbT}
        \kappa_{\eps}^{\lambda}(s)
        D^{2}K_{\eps}(z_{\eps}^{\lambda}(s) - z_{\eps}^{\lambda'}(s'))
        (\mathbf{T}_{\eps}^{\lambda}(s), \mathbf{N}_{\eps}^{\lambda'}(s'))
        (\mathbf{T}_{\eps}^{\lambda'}(s') \cdot \mathbf{N}_{\eps}^{\lambda}(s))^{2}
        \,ds\,ds'\bigg|\,d|\mu|(\lambda'),
    \end{align*}
    which we estimate separately next.
    Here $\overline{\int_{\mathcal{L}}}f(\lambda')\,d|\mu|(\lambda')$
    for  $f\colon\mathcal{L}\to[0,\infty]$ is the upper Lebesgue integral
    \[
        \overline{\int_{\mathcal{L}}}f(\lambda')\,d|\mu|(\lambda')
        \coloneqq \inf_{g}\int_{\mathcal{L}}g(\lambda')\,d|\mu|(\lambda'),
    \]
    where the $\inf$ ranges over all measurable $g\ge f$.

    \textbf{Estimate for $G_{2}$.} Since
    \begin{align*}
        &\frac{\partial}{\partial s'}\left(
            DK_{\eps}(z_{\eps}^{\lambda}(s) - z_{\eps}^{\lambda'}(s'))(\mathbf{T}_{\eps}^{\lambda}(s))
            (\mathbf{T}_{\eps}^{\lambda'}(s') \cdot \mathbf{N}_{\eps}^{\lambda}(s))
            (\mathbf{T}_{\eps}^{\lambda'}(s') \cdot \mathbf{T}_{\eps}^{\lambda}(s))
        \right)
        \\&\quad\quad
        = -D^{2}K_{\eps}(z_{\eps}^{\lambda}(s) - z_{\eps}^{\lambda'}(s'))
        (\mathbf{T}_{\eps}^{\lambda}(s), \mathbf{T}_{\eps}^{\lambda'}(s'))
        (\mathbf{T}_{\eps}^{\lambda'}(s') \cdot \mathbf{N}_{\eps}^{\lambda}(s))
        (\mathbf{T}_{\eps}^{\lambda'}(s') \cdot \mathbf{T}_{\eps}^{\lambda}(s))
        \\&\quad\quad\quad\ 
        + \kappa_{\eps}^{\lambda'}(s')DK_{\eps}(z_{\eps}^{\lambda}(s) - z_{\eps}^{\lambda'}(s'))
        \left(
            (\mathbf{T}_{\eps}^{\lambda'}(s') \cdot \mathbf{T}_{\eps}^{\lambda}(s))^{2}
            - (\mathbf{T}_{\eps}^{\lambda'}(s') \cdot \mathbf{N}_{\eps}^{\lambda}(s))^{2}
        \right),
    \end{align*}
    we see that
    \begin{align*}
       G_{2} \leq C_{\alpha}\overline{\int_{\mathcal{L}}}
        \int_{\ell_{\eps}^{\lambda}\bbT\times\ell_{\eps}^{\lambda'}\bbT}
        \frac{\abs{\kappa_{\eps}^{\lambda}(s)\kappa_{\eps}^{\lambda'}(s')}}
        {\abs{z_{\eps}^{\lambda}(s) - z_{\eps}^{\lambda'}(s')}^{1+2\alpha}}
        \,ds\,ds'd|\mu|(\lambda').
    \end{align*}
    By the Schwarz inequality, the inner integral
    of the right-hand side is bounded by
    \begin{align*}
        \left(
            \int_{\ell_{\eps}^{\lambda}\bbT}\kappa_{\eps}^{\lambda}(s)^{2}
            \int_{\ell_{\eps}^{\lambda'}\bbT}\frac{ds'}
            {\abs{z_{\eps}^{\lambda}(s) - z_{\eps}^{\lambda'}(s')}^{1+2\alpha}}
            \,ds
        \right)^{1/2}
        \left(
            \int_{\ell_{\eps}^{\lambda'}\bbT}\kappa_{\eps}^{\lambda'}(s')^{2}
            \int_{\ell_{\eps}^{\lambda}\bbT}\frac{ds}
            {\abs{z_{\eps}^{\lambda}(s) - z_{\eps}^{\lambda'}(s')}^{1+2\alpha}}
            \,ds'
        \right)^{1/2},
    \end{align*}
    and Lemma~\ref{L3.5} with $\beta=\frac 12$ shows that this is bounded by
    \begin{align*}
        C_{\alpha}\left(
            \norm{z_{\eps}^{\lambda}}_{\dot{H}^{2}}^{2}
            \frac{\ell_{\eps}^{\lambda'}\norm{z_{\eps}^{\lambda'}}_{\dot{H}^{2}}^{2}}
            {(\Delta_{\eps}^{\lambda,\lambda'})^{2\alpha}}
        \right)^{1/2}
        \left(
            \norm{z_{\eps}^{\lambda'}}_{\dot{H}^{2}}^{2}
            \frac{\ell_{\eps}^{\lambda}\norm{z_{\eps}^{\lambda}}_{\dot{H}^{2}}^{2}}
            {(\Delta_{\eps}^{\lambda,\lambda'})^{2\alpha}}
        \right)^{1/2}
        = C_{\alpha}\norm{z_{\eps}^{\lambda}}_{\dot{H}^{2}}^{2}
        \frac{\ell_{\eps}^{\lambda'}\norm{z_{\eps}^{\lambda'}}_{\dot{H}^{2}}^{2} 
        (\ell_{\eps}^{\lambda})^{1/2}}
        {(\ell_{\eps}^{\lambda'})^{1/2}(\Delta_{\eps}^{\lambda,\lambda'})^{2\alpha}}.
    \end{align*}
    Therefore
    \begin{align*}
        \abs{G_{2}} \leq C_{\alpha}R_{\mu}(\Phi_{\eps*}\Theta)
        Q(\Phi_{\eps*}\Theta)\norm{z_{\eps}^{\lambda}}_{\dot{H}^{2}}^{2}.
    \end{align*}

    \textbf{Estimate for $G_{3}$.} We can assume $R_{\mu}(\Phi_{\eps*}\Theta)<\infty$,
    in which case $\Delta_{\eps}^{\lambda,\lambda'} > 0$
    for $|\mu|$-almost all $\lambda'$.  Thus we can apply \cite[Lemma~A.4]{JeoZla2}
    to conclude that
    \begin{equation}\label{3.100}
        G_{3} \leq C_{\alpha}\overline{\int_{\mathcal{L}}}
        \int_{\ell_{\eps}^{\lambda}\bbT\times \ell_{\eps}^{\lambda'}\bbT}
        \frac{\abs{\kappa_{\eps}^{\lambda}(s)}\left(
            \mathcal{M}\kappa_{\eps}^{\lambda}(s)
            + \mathcal{M}\kappa_{\eps}^{\lambda'}(s')
        \right)}
        {\abs{z_{\eps}^{\lambda}(s) - z_{\eps}^{\lambda'}(s')}^{1+2\alpha}}
        \,ds\,ds'd|\mu|(\lambda'),
    \end{equation}
    where $\mathcal{M}$ is the maximal operator given by
    \begin{equation*}
        \mathcal{M}f(s)\coloneqq \max\set{
            \sup_{h\in\left(0,\frac{\ell}{2}\right]}
            \frac{1}{h}\int_{s}^{s+h}\abs{f(s')}\,ds',
            \sup_{h\in\left(0,\frac{\ell}{2}\right]}
            \frac{1}{h}\int_{s-h}^{s}\abs{f(s')}\,ds'
        }
    \end{equation*}
    for $\ell\in(0,\infty)$ and $f\in L^{1}(\ell\bbT)$.
%    from (A.2) in \cite{JeoZla2}.
    Let us split the integrand of \eqref{3.100} into the sum of terms with numerators
    $\abs{\kappa_{\eps}^{\lambda}(s)}\mathcal{M}\kappa_{\eps}^{\lambda}(s)$ and
    $\abs{\kappa_{\eps}^{\lambda}(s)}\mathcal{M}\kappa_{\eps}^{\lambda'}(s')$.  Then
    the maximal inequality 
    %(see (A.3) in \cite{JeoZla2})
    \begin{equation*}
        \norm{\mathcal{M}f}_{L^{2}} \leq C\norm{f}_{L^{2}},
    \end{equation*}
    which holds for all $\ell\in(0,\infty)$ and $f\in L^{2}(\ell\bbT)$
    with some universal constant $C$, shows that  the same argument
    as in the estimate for $G_{2}$ bounds the integral of the second term by
    $C_{\alpha}R_{\mu}(\Phi_{\eps*}\Theta)Q(\Phi_{\eps*}\Theta)
    \norm{z_{\eps}^{\lambda}}_{\dot{H}^{2}}^{2}$.
    As for the first term, Lemma~\ref{L3.5} with $\beta=\frac 12$, Schwarz inequality, and
    the maximal inequality show that
    \begin{align*}
        &\overline{\int_{\mathcal{L}}}
        \int_{\ell_{\eps}^{\lambda}\bbT\times \ell_{\eps}^{\lambda'}\bbT}
        \frac{\abs{\kappa_{\eps}^{\lambda}(s)}\mathcal{M}\kappa_{\eps}^{\lambda}(s)}
        {\abs{z_{\eps}^{\lambda}(s) - z_{\eps}^{\lambda'}(s')}^{1+2\alpha}}
        \,ds\,ds'd|\mu|(\lambda')
        \\&\quad\quad \leq
        \overline{\int_{\mathcal{L}}}
        \int_{\ell_{\eps}^{\lambda}\bbT}\abs{\kappa_{\eps}^{\lambda}(s)}\mathcal{M}\kappa_{\eps}^{\lambda}(s)
        \frac{C_{\alpha}\ell_{\eps}^{\lambda'}\norm{z_{\eps}^{\lambda'}}_{\dot{H}^{2}}^{2}}
        {d(z_{\eps}^{\lambda}(s), \operatorname{im}(z_{\eps}^{\lambda'}))^{2\alpha}}
        \,ds\,d|\mu|(\lambda')
        \\&\quad\quad \leq
        C_{\alpha}Q(\Phi_{\eps*}\Theta)
        \int_{\mathcal{L}}
        \int_{\ell_{\eps}^{\lambda}\bbT}
        \frac{\abs{\kappa_{\eps}^{\lambda}(s)}\mathcal{M}\kappa_{\eps}^{\lambda}(s)}
        {d(z_{\eps}^{\lambda}(s), \operatorname{im}(z_{\eps}^{\lambda'}))^{2\alpha}}
        \,ds\,d|\mu|(\lambda')
        \\&\quad\quad \leq
        C_{\alpha}L_{\mu}(\Phi_{\eps*}\Theta)Q(\Phi_{\eps*}\Theta)
        \int_{\ell_{\eps}^{\lambda}\bbT}
        \abs{\kappa_{\eps}^{\lambda}(s)}\mathcal{M}\kappa_{\eps}^{\lambda}(s)\,ds
        \\&\quad\quad \leq C_{\alpha}L_{\mu}(\Phi_{\eps*}\Theta)Q(\Phi_{\eps*}\Theta)
        \norm{z_{\eps}^{\lambda}}_{\dot{H}^{2}}^{2}.
    \end{align*}
    Here, the regular (rather than upper) integral can be taken after the second inequality
    because the integrand is jointly measurable in $(s,\lambda')$
    (recall from \eqref{1.100} that $d(x,\operatorname{im}(z_{\eps}^{\lambda'}))
    = d(x,\Phi_{\eps}(\partial\Theta^{\lambda'}))$ is jointly measurable
    in $(x,\lambda')$), which was used in the following step when applying Fubini's theorem.
    Aggregating the estimates for $G_{2}$ and $G_{3}$ now yields the desired conclusion.
\end{proof}

Lemmas \ref{L3.1} and \ref{L3.7} suggest existence of an $\eps$-independent estimate
\begin{equation}\label{3.13}
    \max\set{\partial_{t}^{+}Q(\Phi_{\eps*}^{t}\Theta),
    -\partial_{t-}Q(\Phi_{\eps*}^{t}\Theta)}
    \leq C_{\alpha}(L_{\mu}(\Phi_{\eps*}^{t}\Theta) + R_{\mu}(\Phi_{\eps*}^{t}\Theta))
    Q(\Phi_{\eps*}^{t}\Theta)^{2},
\end{equation}
so that a Gr\"{o}nwall-type argument can be used to show that $Q(\Phi_{*}^{t}\Theta)$ is finite for all $t$ near 0.  However, to apply such an argument, we first need to prove that $Q(\Phi_{\eps*}^{t}\Theta)$ is bounded on some ($\eps$-independent) time interval $(-t_0,t_0)$, even if the bound depends on $\eps$, or that $Q(\Phi_{\eps*}^{t}\Theta)$ is upper semi-continuous for each $\eps>0$.
%because of the factor $Q(\Phi_{\eps*}^{t}\Theta)$ in the right-hand side of \eqref{3.11},
%\eqref{3.13} follows from these Lemmas only if we know a priori that
%$Q(\Phi_{\eps*}^{t}\Theta)$ is upper semi-continuous (or locally bounded,
%which implies upper semi-continuity via \eqref{3.11}).
For instance, \eqref{3.11} would become useless if $Q(\Phi_{\eps*}^{t}\Theta)=\infty$ for all small $\eps,|t|>0$, and none of the above a priori excludes this.

%One way to derive upper semi-continuity (or local boundedness) would be
%again through estimating the right-hand side of \eqref{3.2}
%in terms of $\norm{z_{\eps}^{t,\lambda}}_{\dot{H}^{2}}^{2}$.
%In contrast to \eqref{3.11}, such an estimate can depend on $\eps$
%while it must not have the factor $Q(\Phi_{\eps*}^{t}\Theta)$.
%However, the second integral of \eqref{3.12} contains only one factor of $\kappa^{t,\lambda}(s)$,
%so the trivial bound we obtain by estimating $D^{2}(u_{\eps}(\theta_{\eps}^{t}))$
%by its $L^{\infty}$ norm will contain the $L^{1}$ norm of $\kappa^{t,\lambda}$,
%and in order to replace it by the $L^{2}$ norm we have to introduce
%an additional factor of $(\ell^{t,\lambda})^{1/2}$.
%This means that unless we impose an upper on $\ell^{t,\lambda}$, the resulting estimate
%will not be enough for concluding upper semi-continuity (nor local boundedness),
%because it does not rule out the possibility that
%over an arbitrarily small time interval and for an arbitrarily large constant $M$,
%some $\ell^{t,\lambda}\norm{z_{\eps}^{t,\lambda}}_{\dot{H}^{2}}^{2}$ with very large
%$\ell^{t,\lambda}$ grows larger than $M$.
%Another way would be to estimate the $\dot{H}^{2}$ norm of
%$\Phi_{\eps}^{t}\circ\tilde{z}^{0,\lambda}$ where
%$\tilde{z}^{0,\lambda}\colon\bbT\to\bbR^{2}$ is a constant-speed parametrization
%of $z^{0,\lambda}$, but it runs into a similar issue.

It turns out that this problem can be overcome by estimating the second integral in \eqref{3.12} via the Schwarz inequality, provided 
$\sup_{\lambda\in\mathcal{L}}    \ell(z^{0,\lambda})<\infty$ (using also Lemma \ref{L3.1}).  Moreover, we will now show that one can similarly deal with general $\theta^0$ considered in this section via a sequence of approximations satisfying this property.
Take an increasing sequence $\set{\mathcal{L}_{N}'}_{N=1}^{\infty}$ of
measurable subsets of $\mathcal{L}$ such that $|\mu|(\mathcal{L}_{N}')<\infty$
for each $N\in\bbN$ and $\mathcal{L} = \bigcup_{N=1}^{\infty}\mathcal{L}_{N}'$.
For $\theta^0$ as above and any $N\in\bbN$, let
\[
    \mathcal{L}_{N}\,\coloneqq \set{\lambda\in\mathcal{L}_N'\colon
    \ell(z^{0,\lambda}) \leq N},\quad
    \Theta_{N}\coloneqq \Theta\cap(\bbR^{2}\times\mathcal{L}_{N}),\quad
    \theta_{N}^{0}(x)\coloneqq
    \int_{\mathcal{L}_{N}}\mathbbm{1}_{\Theta^{\lambda}}(x)\,d\mu(\lambda).
\]
Then $\theta_{N}^{0} \in L^{1}(\bbR^{2})\cap L^{\infty}(\bbR^{2})$ because
$\norm{\theta_{N}^{0}}_{L^{\infty}} \leq |\mu|(\mathcal{L}_N')$ and
 $\norm{\theta_{N}^{0}}_{L^{1}} \leq \frac{N^{2}}{4\pi}|\mu|(\mathcal{L}_N')$ (the latter
by the isoperimetric inequality), and clearly also $L_{\mu}(\Theta_{N}) \leq L_{\mu}(\Theta)$.
Let $\Phi_{\eps,N}\in C_{\rm loc}\left(\bbR;C(\bbR^{2};\bbR^{2})\right)$ be the corresponding
$\eps$-mollified flow map, that is, the identity map plus the solution to \eqref{2.1}
with $\theta_{N}^{0}$ in place of $\theta^{0}$.
Let $z_{\eps,N}^{t,\lambda}\coloneqq \Phi_{\eps,N}^{t}\circ z^{0,\lambda}$
and $\theta_{\eps,N}^{t} \coloneqq \theta_{N}^{0} \circ (\Phi_{\eps,N}^{t})^{-1}$.

Then \eqref{3.12}, the Schwarz inequality,
$\ell(\gamma)\norm{\gamma}_{\dot{H}^{2}}^{2} \geq 4$
for any $\gamma\in\operatorname{CC}(\bbR^{2})$
(see \cite[Lemma~A.1]{JeoZla2}), Lemma \ref{L3.1}, and \eqref{1.8} show that
\begin{align*}
    \abs{\partial_{t}\norm{z_{\eps,N}^{t,\lambda}}_{\dot{H}^{2}}^{2}}
    &\leq 5\norm{u_{\eps}(\theta_{\eps,N}^{t})}_{\dot{C}^{0,1}}
    \norm{z_{\eps,N}^{t,\lambda}}_{\dot{H}^{2}}^{2}
    + 2\norm{D^{2}(u_{\eps}(\theta_{\eps,N}^{t}))}_{L^{\infty}}
    \ell(z_{\eps,N}^{t,\lambda})^{1/2}\norm{z_{\eps,N}^{t,\lambda}}_{\dot{H}^{2}} \\
    &\leq \left(
        5\norm{D(\nabla^{\perp}K_{\eps})}_{L^{\infty}}
        + C_{\alpha}L_{\mu}(\Theta) N \norm{D^{2}(\nabla^{\perp}K_{\eps})}_{L^{\infty}}
    \right) \norm{\theta_{N}^{0}}_{L^{1}}
    \norm{z_{\eps,N}^{t,\lambda}}_{\dot{H}^{2}}^{2}
\end{align*}
for $(t,\lambda)\in[-T_0,T_0] \times\mathcal{L}_{N}$.  This implies
\[
    \norm{z_{\eps,N}^{t+h,\lambda}}_{\dot{H}^{2}}^{2}
    \leq e^{C\abs{h}}\norm{z_{\eps,N}^{t,\lambda}}_{\dot{H}^{2}}^{2}
\]
whenever $(t,t+h,\lambda)\in [-T_0,T_0]^2\times\mathcal{L}_{N}$, with some 
$C$ depending  on $\alpha,\eps,N,\Theta,\mu$. This and Lemma \ref{L3.1} now yield
%ogether with the inequality
%$\ell(z_{\eps,N}^{t+h,\lambda}) \leq e^{C\abs{h}}\ell(z_{\eps,N}^{t,\lambda})$
%for another $C$ that depends on ${\alpha,\eps,\norm{\theta_{N}^{0}}_{L^{1}}}$,
%which follows from \eqref{3.3} (with $(\theta_{\eps,N},z_{\eps,N})$
%in place of $(\theta_{\eps},z_{\eps})$) and \eqref{1.8}, this shows
\[
    Q(\Phi_{\eps,N*}^{t+h}\Theta_{N}) \leq e^{C\abs{h}}Q(\Phi_{\eps,N*}^{t}\Theta_{N})
\]
for the same $(t,h,\lambda)$, with a new $C$ depending again only on $\alpha,\eps,N,\Theta,\mu$, 
from which we conclude the  $Q(\Phi_{\eps,N*}^{t}\Theta_{N})$ is upper semi-continuous in $t\in[-T_0,T_0]$ for all $N,\eps$ (this then clearly extends to all $t\in\bbR$).

Lemmas \ref{L3.1}, \ref{L3.4}, \ref{L3.7}, and Corollary~\ref{C2.7} with
$(\theta_{\eps,N},\Theta_{N},z_{\eps,N},\Phi_{\eps,N})$ in place of
$(\theta_{\eps},\Theta,z_{\eps}, \Phi_{\eps})$,
and inequalities $L_{\mu}(\Theta_{N})\leq L_{\mu}(\Theta)$,
$R_{\mu}(\Theta_{N})\leq R_{\mu}(\Theta)$, and $Q(\Phi_{\eps,N*}^{t}\Theta_{N}) \geq 4$, show that
\begin{align*}
    \ell(z_{\eps,N}^{t+h,\lambda})\norm{z_{\eps,N}^{t+h,\lambda}}_{\dot{H}^{2}}^{2}
    - Q(\Phi_{\eps,N*}^{t}\Theta_{N})
    &\leq \ell(z_{\eps,N}^{t+h,\lambda})\norm{z_{\eps,N}^{t+h,\lambda}}_{\dot{H}^{2}}^{2}
    - \ell(z_{\eps,N}^{t,\lambda})\norm{z_{\eps,N}^{t,\lambda}}_{\dot{H}^{2}}^{2} \\
    &\leq C_{\alpha}\abs{\int_{t}^{t+h}
    (L_{\mu}(\Theta) + R_{\mu}(\Theta))Q(\Phi_{\eps,N*}^{\tau}\Theta_{N})^{2}\,d\tau}
\end{align*}
holds whenever $(t,t+h,\lambda)\in[-T_{0},T_{0}]^{2}\times\mathcal{L}_{N}$
(with $C_{\alpha}$ depending only on $\alpha$, as always).
Using upper semi-continuity of $Q(\Phi_{\eps,N*}^{t}\Theta_{N})$ in $t$,
taking the supremum over $\lambda\in\mathcal{L}_{N}$,
dividing by $\abs{h}$, and letting $h\to 0$ now yields
\[
    \max\set{
        \partial_{t}^{+}Q(\Phi_{\eps,N*}^{t}\Theta_{N}),
        -\partial_{t-}Q(\Phi_{\eps,N*}^{t}\Theta_{N})
    } \leq C_{\alpha}(L_{\mu}(\Theta) + R_{\mu}(\Theta))Q(\Phi_{\eps,N*}^{t}\Theta_{N})^{2},
\]
so via a Gr\"{o}nwall-type argument we conclude that
\begin{equation}\label{3.14}
    Q(\Phi_{\eps,N*}^{t}\Theta_{N}) \leq \frac{Q(\Theta_{N})}
    {1 - C_{\alpha}(L_{\mu}(\Theta) + R_{\mu}(\Theta))Q(\Theta_{N})\abs{t}}
    \leq \frac{Q(\Theta)}
    {1 - C_{\alpha}(L_{\mu}(\Theta) + R_{\mu}(\Theta))Q(\Theta)\abs{t}}
\end{equation}
whenever $\abs{t} < T_{1}\coloneqq
\frac{1}{C_{\alpha}(L_{\mu}(\Theta) + R_{\mu}(\Theta))Q(\Theta)}$.
%where the second inequality is because $Q(\Theta_{N}) \leq Q(\Theta)$.
We let this $C_\alpha$ be no smaller than the one from Lemma \ref{L2.1} (which was used to define $T_0$), so that $T_1\in(0, T_0]$.

%We now show that $z_{\eps,N} \to z_{\eps}$ a $N\to\infty$, In order 
The following result will allow us to turn \eqref{3.14} into a bound on $Q(\Phi_{*}^{t}\Theta)$.

\begin{lemma}\label{L3.8}
   With $T_{1}$ as above, for each $\eps>0$ we have
    \[
        \lim_{N\to\infty}\sup_{t\in[-T_{1},T_{1}]}\norm{
            \Phi_{\eps,N}^{t} - \Phi_{\eps}^{t}
        }_{L^{\infty}} = 0.
    \]
\end{lemma}

\begin{proof}
    Since $L_{\mu}(\Theta_{N}) \leq L_{\mu}(\Theta)$ for all $N$,
 the last claim in Proposition~\ref{P2.8} and \eqref{1.7} show that
    \[
        M\coloneqq \sup_{(t,N)\in[-T_{1},T_{1}]\times\bbN}
        \max\set{L_{\mu}(\Phi_{\eps,N*}^{t}\theta_{N}^{0}),
        L_{\mu}(\Phi_{\eps*}^{t}\theta_{N}^{0})}
        \leq e^{2\alpha}L_{\mu}(\Theta_{N}) 
        \leq 3L_{\mu}(\Theta).
    \]
    Then  Lemmas \ref{L2.1} and \ref{L2.2} show for any $t\in[-T_{1},T_{1}]$ that
    \begin{equation}\label{3.150}\begin{aligned}
        &\norm{u_{\eps}(\Phi_{\eps,N*}^{t}\theta_{N}^{0})\circ\Phi_{\eps,N}^{t}
        - u_{\eps}(\Phi_{\eps*}^{t}\theta^{0})\circ\Phi_{\eps}^{t}}_{L^{\infty}}
        \\&\quad\quad
        \leq \norm{u_{\eps}(\Phi_{\eps,N*}^{t}\theta_{N}^{0})}_{\dot{C}^{0,1}}
        \norm{\Phi_{\eps,N}^{t} - \Phi_{\eps}^{t}}_{L^{\infty}}
        + \norm{u_{\eps}(\Phi_{\eps,N*}^{t}\theta_{N}^{0})
        - u_{\eps}(\Phi_{\eps*}^{t}\theta_{N}^{0})}_{L^{\infty}}
        \\&\qquad\qquad\qquad\qquad
        + \norm{u_{\eps}(\Phi_{\eps*}^{t}\theta_{N}^{0})
        - u_{\eps}(\Phi_{\eps*}^{t}\theta^{0})}_{L^{\infty}}
        \\&\quad\quad
        \leq C_{\alpha}M\norm{\Phi_{\eps,N}^{t} - \Phi_{\eps}^{t}}_{L^{\infty}}
        + \norm{u_{\eps}(\Phi_{\eps*}^{t}\theta_{N}^{0})
        - u_{\eps}(\Phi_{\eps*}^{t}\theta^{0})}_{L^{\infty}}
    \end{aligned}\end{equation}
    for some $C_{\alpha}$.  We see that $\theta_{N}^{0} \to \theta^{0}$ in $L^{1}$ because
 $\Theta\subseteq \bbR^{2}\times\mathcal L$ has finite measure (with the Lebesgue measure on $\bbR^{2}$
    and $\abs{\mu}$ on $\mathcal L$).
    So for any $\eta>0$ there is $t$-independent $N_{\eta}\in\bbN$  such that for all $N\geq N_{\eta}$ we have
    \begin{align*}
        \norm{u_{\eps}(\Phi_{\eps*}^{t}\theta_{N}^{0})
        - u_{\eps}(\Phi_{\eps*}^{t}\theta^{0})}_{L^{\infty}}
        &= \sup_{x\in\bbR^{2}}\abs{\int_{\bbR^{2}}
        \nabla^{\perp}K_{\eps}(x - \Phi_{\eps}^{t}(y))
        \left(\theta_{N}^{0}(y) - \theta^{0}(y)\right)dy} \\
        &\leq \norm{\nabla^{\perp}K_{\eps}}_{L^{\infty}}
        \norm{\theta_{N}^{0} - \theta^{0}}_{L^{1}} \leq \eta.
    \end{align*}
     Since
    $\norm{\Phi_{\eps,N}^{t}-\Phi_{\eps}^{t}}_{L^{\infty}}$
    is continuous in $t$, \eqref{3.150} yields
    \[
        \max\set{
            \partial_{t}^{+}\norm{\Phi_{\eps,N}^{t} - \Phi_{\eps}^{t}}_{L^{\infty}},
            -\partial_{t-}\norm{\Phi_{\eps,N}^{t} - \Phi_{\eps}^{t}}_{L^{\infty}},
        } \leq C_{\alpha}M\norm{\Phi_{\eps,N}^{t} - \Phi_{\eps}^{t}}_{L^{\infty}} + \eta
    \]
    for all $t\in[-T_{1},T_{1}]$ and  $N\ge N_\eta$.
    Hence, a Gr\"{o}nwall-type argument shows that
    \[
        \norm{\Phi_{\eps,N}^{t} - \Phi_{\eps}^{t}}_{L^{\infty}}
        \leq \begin{cases}
            \frac{\eta}{C_{\alpha}M}\left(e^{C_{\alpha}M\abs{t}} - 1\right)
            & \textrm{if $M>0$}, \\
            \eta\abs{t} & \textrm{if $M=0$}
        \end{cases}
    \]
    for all such $(t,N)$, from which the claim follows.
\end{proof}

Since $\mathcal{L} = \bigcup_{N=1}^{\infty}\mathcal{L}_{N}$,
Lemma~\ref{L3.8} shows that $\lim_{N\to\infty}z_{\eps,N}^{t,\lambda}= z_{\eps}^{t,\lambda}$ in  $\operatorname{CC}(\bbR^{2})$
for any $\eps>0$ and $(t,\lambda)\in(-T_{1},T_{1})\times\mathcal{L}$.
Lower semi-continuity of $\ell(\cdot)$ and $\norm{\cdot}_{\dot{H}^{2}}$ on $\operatorname{CC}(\bbR^{2})$
(the latter by \cite[Corollary~B.3]{JeoZla2}) and \eqref{3.14} show that
\[
    \ell(z_{\eps}^{t,\lambda})\norm{z_{\eps}^{t,\lambda}}_{\dot{H}^{2}}^{2}
    \leq \frac{Q(\Theta)}{1 - C_{\alpha}(L_{\mu}(\Theta) + R_{\mu}(\Theta))Q(\Theta)\abs{t}}
\]
holds for these $(\eps,t,\lambda)$.
%\in[-T_{1},T_{1}]\times\mathcal{L}$.
%(recall $T_{1}\coloneqq \frac{1}{C_{\alpha}(L_{\mu}(\Theta) + R_{\mu}(\Theta))Q(\Theta)}$
%and $T_{1}\leq T_{0}$ since $Q(\Theta)\geq 4$).
After taking $\eps\to 0$ and using the same lower semi-continuity again, we obtain the same bound with $z^{t,\lambda}$ in place of $z_{\eps}^{t,\lambda}$.  It follows that
%Taking the supremum over $\lambda$ finally establishes
\[
    Q(\Phi_{*}^{t}\Theta)
    \leq \frac{Q(\Theta)}{1 - C_{\alpha}(L_{\mu}(\Theta) + R_{\mu}(\Theta))Q(\Theta)\abs{t}}
\]
for all $t\in(-T_{1},T_{1})$.  Since $L_{\mu}(\Phi^t_*\Theta)$ and  $R_{\mu}(\Phi^t_*\Theta)$ are bounded on compact subsets of $I$, an analogous bound holds for all $t$ near any $\tau\in I$ with $Q(\Phi^\tau_*\Theta)<\infty$ (with $(\Phi^\tau_*\Theta,|t-\tau|)$ in place of $(\Theta,|t|)$).  This proves Theorem~\ref{T1.4}(ii).

\section{Proof of Theorem~\ref{T1.4}(iii)}\label{S4}

Once again, all constants $C_{\alpha}$ below
can change from one inequality to another, but they always only depend  on $\alpha$.
Let 
\[
\Sigma(\Theta):= \inf_{\lambda\in\mathcal{L}}\abs{\Theta^{\lambda}}>0 \qquad\text{and}\qquad \Lambda(\Theta) \coloneqq \int_{\mathcal{L}}\ell(\partial\Theta^{\lambda})\,d|\mu|(\lambda)<\infty.
\]
We now also assume that $\alpha\in\left(0,\frac{1}{6}\right]$,
$|\mu|(\mathcal L)<\infty$, and  $T$ is an endpoint of the interval $I'$ with $|T|<\infty$
 and $\limsup_{t\to T}Q(\Phi_{*}^{t}\Theta) < \infty$.  We let  $S:=\sup_{t\in[0,T)}Q(\Phi_{*}^{t}\Theta) < \infty$ and will then obtain a contradiction.
 % (towards contradiction).

Since $\ell(\cdot)$ and $\norm{\cdot}_{\dot{H}^{2}}$ are lower semi-continuous
on $\operatorname{CC}(\bbR^{2})$ (see the last paragraph of Section~\ref{S3}),  $Q(\Phi_{*}^{t}\Theta)$ is lower semi-continuous in $t$.
%because it is the supremum of lower semi-continuous functions
%$\ell(z^{t,\lambda})\norm{z^{t,\lambda}}_{\dot{H}^{2}}^{2}$ for $\lambda\in\mathcal{L}$
%(recall that $\ell(\cdot)$ and $\norm{\cdot}_{\dot{H}^{2}}$ are lower semi-continuous
%on $\operatorname{CC}(\bbR^{2})$, where the latter is by
%\cite[Corollary B.3]{JeoZla2}). 
So if $T$ is not an endpoint of $I$, we obtain $Q(\Phi_{*}^{T}\Theta)<\infty$ and hence a contradiction with Theorems \ref{T1.2} and \ref{T1.4}(ii).  Therefore 
 $T$ is an endpoint of $I$, and the desired contradiction will follow once we
 % as well, because otherwise
%the definition of $I'$ and Theorem~\ref{T1.4}(ii)
%applied to the inital datum $\theta^{T}$ at the initial time $T$ yield a contradiction.
show  $\limsup_{t\to T}L_{\mu}(\Phi_{*}^{t}\Theta) < \infty$.
%which contradicts to Theorem~\ref{T1.2}.
We can assume that $T>0$ without loss because the other case is identical.

\begin{lemma}\label{L4.1}
For some $C_\alpha$ and all $(t,\lambda)\in [0,T)\times\mathcal L$ we have
    \begin{equation}\label{4.1}
        \ell(z^{t,\lambda}) \leq
        S \left(
            \frac{\ell(z^{0,\lambda})}{2} +
            C_{\alpha}T\left(
                \norm{\theta^{0}}_{L^{1}} + \norm{\theta^{0}}_{L^{\infty}}
            \right)
        \right) .
    \end{equation}
\end{lemma}

\begin{proof}
    Integration by parts and the Schwarz inequality show that for any
    rectifiable $\gamma\in\operatorname{CC}(\bbR^{2})$ we have
    \[
        \ell(\gamma) \leq \operatorname{diam}(\operatorname{im}(\gamma))^{2}
        \norm{\gamma}_{\dot{H}^{2}}^{2} 
        \le \operatorname{diam}(\operatorname{im}(\gamma))\, \ell(\gamma) \norm{\gamma}_{\dot{H}^{2}}^{2} .
    \]
    Then it is enough to prove that for all $(t,\lambda)\in [0,T)\times\mathcal L$ we have
    \begin{equation}\label{4.2}
        \partial_{t}^{+}\operatorname{diam}(\operatorname{im}(z^{t,\lambda}))
        \leq C_{\alpha}\left(\norm{\theta^{0}}_{L^{1}} + \norm{\theta^{0}}_{L^{\infty}}\right)
    \end{equation}
    because then this and $\operatorname{diam}(\operatorname{im}(\gamma)) \leq \frac{\ell(\gamma)}2$
    %, and  a Gr\"{o}nwall-type argument 
    yield  \eqref{4.1}.

    Let $0\leq t\leq t+h<T$ and pick any $x,y\in\operatorname{im}(z^{0,\lambda})$
    such that $\abs{\Phi^{t+h}(x) - \Phi^{t+h}(y)}
    = \operatorname{diam}(\operatorname{im}(z^{t+h,\lambda}))$. Then clearly
    \begin{equation*} %\label{4.3}
    \begin{aligned}
        \operatorname{diam}(\operatorname{im}(z^{t+h,\lambda}))
        - \operatorname{diam}(\operatorname{im}(z^{t,\lambda}))
        &\leq \abs{\Phi^{t+h}(x) - \Phi^{t+h}(y)} - \abs{\Phi^{t}(x) - \Phi^{t}(y)} \\
        &\leq \int_{t}^{t+h} \left(\abs{u(\theta^{\tau};\Phi^{\tau}(x))}
        + \abs{u(\theta^{\tau};\Phi^{\tau}(y))} \right)d\tau,
    \end{aligned}\end{equation*}
    and for any $\tau\in[0,T)$ we have
    \begin{equation*}
        \norm{u(\theta^{\tau})}_{L^{\infty}}
        \leq C_{\alpha}\left(
            \norm{\theta^{\tau}}_{L^{1}} + \norm{\theta^{\tau}}_{L^{\infty}}
        \right)
        = C_{\alpha}\left(
            \norm{\theta^{0}}_{L^{1}} + \norm{\theta^{0}}_{L^{\infty}}
        \right)
    \end{equation*}
    for some $C_{\alpha}$ since $\Phi^{\tau}$ is measure-preserving.
    This now indeed yields \eqref{4.2}.
\end{proof}

Next,
%we have $\sup_{t\in[0,T)}Q(\Phi_{*}^{t}\Theta) < \infty$, so 
the isoperimetric inequality and $\Phi^{t}$ being measure-preserving show that
\begin{align*}
    M\coloneqq \sup_{t\in[0,T)}\sup_{\lambda\in\mathcal{L}}
    \norm{z^{t,\lambda}}_{\dot{H}^{2}}^{2}
    & \leq \sup_{t\in[0,T)}\frac{Q(\Phi_{*}^{t}\Theta)}
    {\inf_{\lambda\in\mathcal{L}}\ell(z^{t,\lambda})}
     \leq \sup_{t\in[0,T)}\frac{Q(\Phi_{*}^{t}\Theta)}{2\sqrt{\pi}\,\Sigma(\Phi_{*}^{t}\Theta)^{1/2}}
    \\ &    = \sup_{t\in[0,T)}\frac{Q(\Phi_{*}^{t}\Theta)}{2\sqrt{\pi}\,\Sigma(\Theta)^{1/2}} < \infty.
\end{align*}
%Note that the last equality is because each $\Phi^{t}$ is measure-preserving.
For any $\eta>0$ and any $\tilde{\Theta}\subseteq\bbR^{2}\times\mathcal{L}$
from the product $\sigma$-algebra define
\[
    L_{\mu}^{\eta}(\tilde{\Theta}) \coloneqq
    \sup_{x\in\bbR^{2}}\int_{\mathcal{L}}\frac{d|\mu|(\lambda)}
    {\left(d(x,\partial\tilde{\Theta}^{\lambda}) + \eta\right)^{2\alpha}}.
\]
Then from $\Phi-{\rm Id}\in C\left([0,T);BC(\bbR^{2};\bbR^{2})\right)$ (see Proposition~\ref{P2.8}) we see that
%a similar proof as in Proposition~\ref{P2.6} shows that
$L_{\mu}^{\eta}(\Phi_{*}^{t}\Theta)$ is continuous in $t\in[0,T)$.
Since $\lim_{\eta\to 0^{+}} L_{\mu}^{\eta}(\tilde{\Theta}) = L_{\mu}(\tilde{\Theta})$,
 it will suffice to obtain the following $\eta$-independent upper bound on
$L_{\mu}^{\eta}(\Phi_{*}^{t}\Theta)$.

\begin{lemma}\label{L4.2}
    There is $C_{\alpha}$ such that for any $\eta>0$ and any $t\in[0,T)$ we have
    \begin{equation}\label{4.4}
        \partial_{t}^{+}L_{\mu}^{\eta}(\Phi_{*}^{t}\Theta)
        \leq C_{\alpha}SM^{1+2\alpha}\left(
            \Lambda(\Theta) + T|\mu|(\mathcal L) \left(   \norm{\theta^{0}}_{L^{1}} +  \norm{\theta^{0}}_{L^{\infty}} \right)
        \right)
        L_{\mu}^{\eta}(\Phi_{*}^{t}\Theta).
    \end{equation}
\end{lemma}

\begin{proof}
    Fix $\eta>0$ and $t\in[0,T)$, and find a decreasing sequence $t_{n}\to t$ such that
    \beq\lb{11.101}
        \partial_{t}^{+}L_{\mu}^{\eta}(\Phi_{*}^{t}\Theta)
        = \lim_{n\to\infty}\frac{L_{\mu}^{\eta}(\Phi_{*}^{t_{n}}\Theta)
        - L_{\mu}^{\eta}(\Phi_{*}^{t}\Theta)}{t_{n} - t}.
    \eeq
    Without loss assume that $t_{1}\leq t + \frac{1}{2C_{\alpha}L_{\mu}(\Phi_{*}^{t}\Theta)}$
    with $C_{\alpha}$ from Lemma~\ref{L2.1}, so that
    \begin{equation}\label{4.5}
        L_{\mu}(\Phi_{*}^{\tau}\theta)\leq 2L_{\mu}(\Phi_{*}^{t}\Theta)
    \end{equation}
    holds for all $\tau\in [t,t_{1}]$ by Proposition~\ref{P2.8}.
    Fix any $n\in\bbN$ and pick $x\in\bbR^{2}$ such that
    \beq\lb{11.100}
        L_{\mu}^{\eta}(\Phi_{*}^{t_{n}}\Theta)
        \leq \int_{\mathcal{L}}\frac{d|\mu|(\lambda)}
        {\left(
            d(\Phi^{t_{n}}(x), \operatorname{im}(z^{t_{n},\lambda})) + \eta
        \right)^{2\alpha}}
        + \frac{t_{n} - t}{n}.
    \eeq
    Then fix any $\lambda\in\mathcal{L}$ and pick
    $y\in\operatorname{im}(z^{0,\lambda})$ such that
    \[
        d(\Phi^{t_{n}}(x), \operatorname{im}(z^{t_{n},\lambda}))
        = \abs{\Phi^{t_{n}}(x) - \Phi^{t_{n}}(y)}.
    \]

    Since $L_{\mu}^{\eta}(\Phi_{*}^{t_{n}}\Theta)
    \leq L_{\mu}(\Phi_{*}^{t_{n}}\Theta) <\infty$,
    for $|\mu|$-almost every $\lambda$ we have
    $\abs{\Phi^{t_{n}}(x) - \Phi^{t_{n}}(y)} > 0$ and so
 $x\neq y$. For any such $\lambda$,
    an argument similar to \eqref{2.4} yields the first inequality in
    \begin{align*}
        &\frac{1}{\left(
            d(\Phi^{t_{n}}(x), \operatorname{im}(z^{t_{n},\lambda})) + \eta
        \right)^{2\alpha}}
        - \frac{1}{\left(
            d(\Phi^{t}(x), \operatorname{im}(z^{t,\lambda})) + \eta
        \right)^{2\alpha}}
        \\&\qquad\quad\leq
        \frac{\abs{\Phi^{t}(x) - \Phi^{t}(y)}
        - \abs{\Phi^{t_{n}}(x) - \Phi^{t_{n}}(y)}}
        {\left(
            \abs{\Phi^{t_{n}}(x) - \Phi^{t_{n}}(y)} + \eta
        \right)
        \left(
            \abs{\Phi^{t}(x) - \Phi^{t}(y)} + \eta
        \right)^{2\alpha}}
        \\&\qquad\quad\leq
        \frac{\Phi^{t}(x) - \Phi^{t}(y)}
        {\abs{\Phi^{t}(x) - \Phi^{t}(y)}}\cdot
        \frac{(\Phi^{t}(x) - \Phi^{t_{n}}(x))
        - (\Phi^{t}(y) - \Phi^{t_{n}}(y))}
        {\left(
            \abs{\Phi^{t_{n}}(x) - \Phi^{t_{n}}(y)} + \eta
        \right)
        \left(
            \abs{\Phi^{t}(x) - \Phi^{t}(y)} + \eta
        \right)^{2\alpha}}.
    \end{align*}
    We now write the right-hand side as the sum of the following terms
    which we estimate separately:
    \begin{align*}
        P_{1}^n &\coloneqq \frac{\Phi^{t_{n}}(x) - \Phi^{t_{n}}(y)}
        {\abs{\Phi^{t_{n}}(x) - \Phi^{t_{n}}(y)}}\cdot
        \frac{(t_{n} - t)\left(
            u(\theta^{t_{n}};\Phi^{t_{n}}(y))
            - u(\theta^{t_{n}};\Phi^{t_{n}}(x))
        \right)}
        {\left(
            \abs{\Phi^{t_{n}}(x) - \Phi^{t_{n}}(y)} + \eta
        \right)
        \left(
            \abs{\Phi^{t}(x) - \Phi^{t}(y)} + \eta
        \right)^{2\alpha}}, \\
        P_{2}^n &\coloneqq \left[
            \frac{\Phi^{t}(x) - \Phi^{t}(y)}
            {\abs{\Phi^{t}(x) - \Phi^{t}(y)}}
            - \frac{\Phi^{t_{n}}(x) - \Phi^{t_{n}}(y)}
            {\abs{\Phi^{t_{n}}(x) - \Phi^{t_{n}}(y)}}
        \right] \cdot
         \frac{(t_{n} - t)\left(
            u(\theta^{t_{n}};\Phi^{t_{n}}(y))
            - u(\theta^{t_{n}};\Phi^{t_{n}}(x))
        \right)}
        {\left(
            \abs{\Phi^{t_{n}}(x) - \Phi^{t_{n}}(y)} + \eta
        \right)
        \left(
            \abs{\Phi^{t}(x) - \Phi^{t}(y)} + \eta
        \right)^{2\alpha}}, \\
        P_{3}^n &\coloneqq
        \frac{\Phi^{t}(x) - \Phi^{t}(y)}
        {\abs{\Phi^{t}(x) - \Phi^{t}(y)}} \cdot
        \frac{
            \int_{t}^{t_{n}}
            \left[ u(\theta^{\tau};\Phi^{\tau}(y))
            - u(\theta^{t_{n}};\Phi^{t_{n}}(y))
            - u(\theta^{\tau};\Phi^{\tau}(x))
            + u(\theta^{t_{n}};\Phi^{t_{n}}(x))
            \right] d\tau
        }{\left(
            \abs{\Phi^{t_{n}}(x) - \Phi^{t_{n}}(y)} + \eta
        \right)
        \left(
            \abs{\Phi^{t}(x) - \Phi^{t}(y)} + \eta
        \right)^{2\alpha}}.
    \end{align*}

    \textbf{Estimate for $P_{1}^n$.} Since $|\mu|(\mathcal L)<\infty$ and
    $\nabla^{\perp}K(\Phi^{t_{n}}(y) - \,\cdot\,)
    - \nabla^{\perp}K(\Phi^{t_{n}}(x) - \,\cdot\,)$ is integrable,
    Fubini's theorem and Green's theorem show that
    \begin{align*}
        &u(\theta^{t_{n}};\Phi^{t_{n}}(y))
        - u(\theta^{t_{n}};\Phi^{t_{n}}(x)) \\
        &\quad\ = \int_{\bbR^{2}}\left(
            \nabla^{\perp}K(\Phi^{t_{n}}(y) - y')
            - \nabla^{\perp}K(\Phi^{t_{n}}(x) - y')
        \right)
        \theta^{t_{n}}(y')\,dy' \\
        &\quad\ = \int_{\mathcal{L}}\int_{\Phi^{t_{n}}(\Theta^{\lambda'})}\left(
            \nabla^{\perp}K(\Phi^{t_{n}}(y) - y')
            - \nabla^{\perp}K(\Phi^{t_{n}}(x) - y')
        \right)dy'\,d\mu(\lambda') \\
        &\quad\ = \int_{\mathcal{L}}\int_{\ell(z^{t_{n},\lambda'})\bbT}\left(
            K(\Phi^{t_{n}}(y) - z^{t_{n},\lambda'}(s))
            - K(\Phi^{t_{n}}(x) - z^{t_{n},\lambda'}(s))
        \right)\partial_{s}z^{t_{n},\lambda'}(s)\,ds\,d\mu(\lambda')
    \end{align*}
    where $z^{t_{n},\lambda'}(\cdot)$ is an arbitrary arclength parametrization of
    $z^{t_{n},\lambda'}$.
    Moreover, the definition of $y$ shows that $\Phi^{t_{n}}(x) - \Phi^{t_{n}}(y)$
    is normal to $z^{t_{n},\lambda}$ at $\Phi^{t_{n}}(y)$, and  from $R_{\mu}(\Phi_{*}^{t_{n}}\Theta)<\infty$
    we know that $\operatorname{im}(z^{t_{n},\lambda})
    \cap \operatorname{im}(z^{t_{n},\lambda'}) = \emptyset$ for
    $\abs{\mu}$-almost every $\lambda'$. Then for each such $\lambda'$,
    \cite[Lemma 2.7]{JeoZla2} (with $\beta\coloneqq \frac 12$ and
    $(\gamma_{1},\gamma_{2})\coloneqq (z^{t_{n},\lambda}, z^{t_{n},\lambda'})$)
    and Lemma~\ref{L4.1} show that
    \begin{align*}
        &\abs{\left(
            u(\theta^{t_{n}};\Phi^{t_{n}}(y))
            - u(\theta^{t_{n}};\Phi^{t_{n}}(x))
        \right)\cdot \frac{\Phi^{t_{n}}(x) - \Phi^{t_{n}}(y)}
        {\abs{\Phi^{t_{n}}(x) - \Phi^{t_{n}}(y)}}} \\
        &\qquad\quad \leq
        C_{\alpha}M^{1+2\alpha}
        \abs{\Phi^{t_{n}}(y) - \Phi^{t_{n}}(x)}
        \int_{\mathcal{L}}\ell(z^{t_{n},\lambda'})\,d|\mu|(\lambda') \\
        &\qquad\quad \leq
        C_{\alpha}SM^{1+2\alpha}\left(
            \Lambda(\Theta) + T|\mu|(\mathcal L) \left(   \norm{\theta^{0}}_{L^{1}} +  \norm{\theta^{0}}_{L^{\infty}} \right)
        \right)
        \abs{\Phi^{t_{n}}(y) - \Phi^{t_{n}}(x)}
    \end{align*}
    for some constant $C_\alpha$.  This then yields
    \begin{equation}\label{4.6}
        \abs{P_{1}^n} \leq \frac{ C_{\alpha}SM^{1+2\alpha}\left(
            \Lambda(\Theta) + T|\mu|(\mathcal L) \left(   \norm{\theta^{0}}_{L^{1}} +  \norm{\theta^{0}}_{L^{\infty}} \right)
        \right)   }
        {\left(            d(\Phi^{t}(x), \operatorname{im}(z^{t,\lambda})) + \eta
        \right)^{2\alpha}}  (t_{n} - t) .
    \end{equation}

    \textbf{Estimate for $P_{2}^n$.} By the same argument as in \eqref{3.6} we have
    \[
        \abs{\Phi^{\tau}(x) - \Phi^{\tau}(y)
        - \Phi^{t_{n}}(x) + \Phi^{t_{n}}(y)}
        \leq \left(
            e^{2C_{\alpha}L_{\mu}(\Phi_{*}^{t}\Theta)(t_{n}-\tau)} - 1
        \right)
        \abs{\Phi^{t_{n}}(x) - \Phi^{t_{n}}(y)}
    \]
    for any $\tau\in[t,t_{n}]$. Therefore
    \begin{align*}
        &\abs{\frac{\Phi^{t}(x) - \Phi^{t}(y)}
            {\abs{\Phi^{t}(x) - \Phi^{t}(y)}}
            - \frac{\Phi^{t_{n}}(x) - \Phi^{t_{n}}(y)}
            {\abs{\Phi^{t_{n}}(x) - \Phi^{t_{n}}(y)}}
        } \\
        &\qquad\leq
                \frac{\abs{\Phi^{t}(x) - \Phi^{t}(y)
        - \Phi^{t_{n}}(x) + \Phi^{t_{n}}(y)}}
        {\abs{\Phi^{t_{n}}(x) - \Phi^{t_{n}}(y)}} +
        \frac{\abs{\abs{\Phi^{t_{n}}(x) - \Phi^{t_{n}}(y)}
        - \abs{\Phi^{t}(x) - \Phi^{t}(y)}}}
        {        \abs{\Phi^{t_{n}}(x) - \Phi^{t_{n}}(y)}}
\\
        &\qquad\leq
        2\left(
            e^{2C_{\alpha}L_{\mu}(\Phi_{*}^{t}\Theta)(t_{n}-\tau)} - 1
        \right),
    \end{align*}
    which together with Lemma~\ref{L2.1} and \eqref{4.5} yields
    \begin{equation}\label{4.7}
    %\begin{aligned}
        \abs{P_{2}^n} 
%        &\leq \frac{2\left(
%            e^{2C_{\alpha}L_{\mu}(\Phi_{*}^{t}\Theta)(t_{n}-t)} - 1
%        \right)\norm{u(\theta^{t_{n}})}_{\dot{C}^{0,1}}}
%        {\left(            d(\Phi^{t}(x), \operatorname{im}(z^{t,\lambda})) + \eta
%        \right)^{2\alpha}} (t_{n} - t) \\
        \leq \frac{C_{\alpha}\left(
            e^{2C_{\alpha}L_{\mu}(\Phi_{*}^{t}\Theta)(t_{n}-t)} - 1
        \right)  L_{\mu}(\Phi_{*}^{t}\Theta) }
        {\left(
            d(\Phi^{t}(x), \operatorname{im}(z^{t,\lambda})) + \eta
        \right)^{2\alpha}} (t_{n} - t) .
%    \end{aligned}
\end{equation}

    \textbf{Estimate for $P_{3}^n$.} By Lemma~\ref{L2.1}, Lemma~\ref{L2.2}, and \eqref{4.5} we have
    \begin{align*}
        &\int_{t}^{t_{n}}
        \abs{
            u(\theta^{\tau};\Phi^{\tau}(y))
            - u(\theta^{t_{n}};\Phi^{t_{n}}(y))
            - u(\theta^{\tau};\Phi^{\tau}(x))
            + u(\theta^{t_{n}};\Phi^{t_{n}}(x))
        }d\tau \\
        &\qquad\qquad\leq
        2\int_{t}^{t_{n}} \left(
         \norm{u(\Phi_{*}^{\tau}\Theta) - u(\Phi_{*}^{t_{n}}\Theta)}_{L^{\infty}}  +
        \norm{u(\theta^{\tau})}_{\dot{C}^{0,1}}
        \norm{\Phi^{\tau} - \Phi^{t_{n}}}_{L^{\infty}} \right)
        d\tau \\
        &\qquad\qquad\leq
        C_{\alpha}L_{\mu}(\Phi_{*}^{t}\Theta)
        \int_{t}^{t_{n}}\norm{\Phi^{\tau} - \Phi^{t_{n}}}_{L^{\infty}}\,d\tau,
    \end{align*}
    so
    \begin{equation}\label{4.8}
        \abs{P_{3}^n} \leq \frac{C_{\alpha}L_{\mu}(\Phi_{*}^{t}\Theta)
        \int_{t}^{t_{n}}\norm{\Phi^{\tau} - \Phi^{t_{n}}}_{L^{\infty}}\,d\tau}
        {\eta\left(
            d(\Phi^{t}(x),\operatorname{im}(z^{t,\lambda})) + \eta
        \right)^{2\alpha}}.
    \end{equation}
    \smallskip

Estimates \eqref{4.7} and \eqref{4.8} show that $\lim_{n\to\infty} \frac{|P_j^n|}{t_n-t}=0$ uniformly in $\lambda\in\mathcal L$ for $j=2,3$, and then \eqref{11.101}, \eqref{11.100}, and  \eqref{4.6} yield \eqref{4.4}.
\end{proof}

Lemma~\ref{L4.2} and a Gr\"{o}nwall-type argument now show
\[
    L_{\mu}^{\eta}(\Phi_{*}^{t}\Theta)
    \leq e^{M_{1}T}L_{\mu}^{\eta}(\Theta)
    \leq e^{M_{1}T}L_{\mu}(\Theta)
\]
for any $t\in[0,T)$, where
\[
    M_{1} \coloneqq C_{\alpha}SM^{1+2\alpha}\left(
            \Lambda(\Theta) + T|\mu|(\mathcal L) \left(   \norm{\theta^{0}}_{L^{1}} +  \norm{\theta^{0}}_{L^{\infty}} \right)
        \right)
\]
and $C_{\alpha}$ is from Lemma~\ref{L4.2}. Letting $\eta\to 0$
then yields
\[
    \limsup_{t\to T}L_{\mu}(\Phi_{*}^{t}\Theta)
    \leq e^{M_{1}T}L_{\mu}(\Theta) < \infty,
\]
 which is the desired  contradiction and finishes the proof of Theorem~\ref{T1.4}(iii).

\end{document}